\documentclass[review]{elsarticle}

\usepackage[latin9]{inputenc}
\usepackage[english]{babel}
\usepackage{amstext,amsmath,amsfonts,amssymb,amsbsy,bm,amsthm}
\usepackage[usenames,dvipsnames]{color}
\usepackage{hyperref}
\usepackage{algorithmic}
\usepackage{ocr} 
\usepackage{color}
\usepackage{mathtools}
\usepackage[lined,boxed,commentsnumbered]{algorithm2e}

\numberwithin{equation}{section}

\newtheorem{thm}{Theorem}[section]

\def\B{\color{black}}

\usepackage{listings}

\lstnewenvironment{code}[1][]{%
   \lstset{%
      language=C++,
      basicstyle=\ttfamily\footnotesize,
      keywordstyle=\bfseries\color{blue},
      identifierstyle=\color{black},
      stringstyle=\color{red},
      commentstyle=\color{orange},
      morecomment=[l][\color{magenta}]{\#},
      captionpos=b,frame=single,
      numbers=left,                   
      escapeinside={\%*}{*)},
      #1,
   }
}{}

\begin{document}

\begin{frontmatter}

\title{Extreme learning machine collocation for the numerical solution of elliptic PDEs with sharp gradients}

\author[nap]{Francesco Calabr\`o} \ead{calabro@unina.it}
\author[nap2]{Gianluca Fabiani} \ead{gianluca.fabiani@unina.it}
\author[nap,nap2]{Constantinos Siettos} \ead{constantinos.siettos@unina.it}

\address[nap]{Dipartimento di Matematica e Applicazioni ``Renato Caccioppoli", Universit\`a degli Studi di Napoli Federico II }
\address[nap2]{Scuola Superiore Meridionale di Napoli, Universit\`a degli Studi di Napoli Federico II}

\begin{abstract}
We introduce a new numerical method based on machine learning to approximate the solution of elliptic partial differential equations with collocation using a set of sigmoidal functions. We show that a feedforward neural network with a single hidden layer with sigmoidal functions and fixed, random, internal weights and biases can be used to compute accurately a collocation solution. The choice to fix internal weights and bias leads to the so-called Extreme Learning Machine network. We discuss how to determine the range for both internal weights and biases in order to obtain a good underlining approximating space, and we explore the required number of collocation points. 
We demonstrate the efficiency of the proposed method with several one-dimensional diffusion-advection-reaction problems that exhibit steep behaviors, such as boundary layers.
The boundary conditions are imposed directly as collocation equations. We point out that there is no need of training the network, as the proposed numerical approach results to a linear problem that can be easily solved using least-squares. Numerical results show that the proposed method achieves a good accuracy. Finally, we compare the proposed method with finite differences and point out the significant improvements in terms of computational cost, thus avoiding the time-consuming training phase.

\end{abstract}


\end{frontmatter}

\section{Introduction}

Over the last few years various algorithms have been addressed for the numerical solution of Partial Differential Equations (PDEs)  with the aid of Artificial Neural Networks (ANN) (see e.g. \cite{berg2018unified,han2018solving,rudd2015constrained,sirignano2018dgm}). The use of ANN for the numerical solutions of systems of differential equations goes back to \cite{lagaris1998artificial,lagaris2000neural}, and the interest on new applications is in nowadays at the forefront of research interest \cite{anitescu2019artificial,arbabi2020linking,chan2019machine,DWIVEDI202096,guo2020analysis,HAN2020109672,lee2020coarse,lu2019deepxde,michoski2020solving,pang2019neural, SIRIGNANO2020109811}. The main focus of such techniques is on the numerical solution of some very difficult problems: for example high dimensional systems with difficult geometries and non-linear. In all cases, ANN seem to outperform conventional techniques being extremely flexible and computationally fast. The celebrated theorem of the universal approximation of ANNs gives some justification to such behavior, although some results are still unavailable: i.e. it is proved that the approximation error vanishes asymptotically, but there are few available results on explicit error bounds for the approximation error or about  polynomial reproduction (see e.g.  \cite{almira2020negative,costarelli2013approximation,guliyev2018approximation,pinkus1999approximation}). Among all types of ANNs, recently great attention is given to the so-called Extreme Learning Machines (ELM) \cite{HUANG201532,huang2004extreme}. However, the efficiency and applicability of ELM for the approximation of differential problems is still unexplored. Indeed, to the best of our knowledge, the only study on the subject is that of \cite{DWIVEDI202096} where the authors however report a failure of ELM  to deal with sharp gradient problems.\par
Here, we show how ELM can be exploited to deal with boundary problems, in particular 1D second order elliptic linear equations that exhibit steep behaviors, such as boundary layers. It is known that such problems if solved by classical methods such as finite differences may lead to several numerical instabilities such as spurious oscillations (see e.g. \cite{de2010interior,onate2016accurate}) even if the solution is regular. Thus, we focus on linear advection-reaction problems with constant coefficients where the solution is known analytically and it is regular, but the advection-dominated or the reaction-dominated cases may result to boundary layers. For such problems, several approaches (for example artificial diffusion, upwind schemes or mesh adaptivity \cite{MITCHELL2013350,quarteroniLibro}) have been addressed to deal with the emerged steep gradients. Compared to the above techniques, our approach is free from problem-dependent modifications being able to detect the steep gradients that arise.
The results reveal how ELM can serve as a robust numerical approach to obtain accurate solutions without the need of stabilization techniques. In particular, we propose an under-determined collocation method applied to an ELM with one hidden layer which they don't need to be trained. Our findings reveal the ability of ELM to approximate the solutions efficiently, thus extending and giving a new insight on the use of such types of ANN. What we propose here, by taking both internal weights and biases a-priori fixed is that we do not need to train the network in order to obtain a sufficiently accurate solution, at least if the number of neurons fits the scale of the layer. Obviously, such solutions can be taken as initial choices for a training process if the problem is non-linear and/or demands a higher approximation accuracy.
The paper is organized as follows. In Section 2, we introduce our approximation space on the base of ANN and ELM, and in Section 3 we present the proposed ELM collocation method. The approximation efficiency and convergence property of the scheme is analysed in Section 4. In Section 5, we present the numerical results obtained the proposed approach for the solution of several 1D boundary-layer elliptic problems and compare them with both the exact-analytical and FD numerical solution.

\section{The proposed ANN architecture: the ELM network}
In this section, we briefly introduce the ANN that we propose and the main results that give the properties of such universal approximators. We consider an ANN with a single hidden layer with $n$ neurons. We choose the activation function to be the standard logistic sigmoid function defined by
\begin{equation*}
\sigma(x)=\frac{1}{1+\text{exp}(-x)} \ .
\label{eq:sigmoide}
\end{equation*}
The output function is the linear activation function, so that the final approximation is achieved via a linear combination of sigmoid functions. In the network one has to fix the biases $\beta$, the connections between the input layer and the hidden layer ${\alpha}$, represented by a vector of internal weights, and the connections of the hidden layer with the output layer ${w}$, denoted by the external weights. Therefore, this type of ANN describes a map $G:\mathbb{R}\rightarrow \mathbb{R}$ that can be written as:
\begin{equation}
G(x; \bm{\alpha},\bm{\beta},\bm{w})=\sum_{i=1}^n w_i \sigma(\alpha_i x + \beta_i).
\label{eq:ANNfun}
\end{equation}
As it is well-known, this kind of ANN is a universal approximator for any $L^1$ function 
as stated, for example in \cite{cybenko1989approximation, hornik1990universal,hornik1989multilayer}.
\begin{thm}
For any function $f \in L^1([a,b])$ and for all $\varepsilon>0$ there exist a choice of $\bm{\alpha},\bm{\beta},\bm{w}$  such that
\begin{equation*}
\|G-f\|_{L^1} < \varepsilon.
\end{equation*}
where $\| \cdot \|_{L^1}$ is the usual $L^1$ norm, i.e. $||\psi||_{L^1}=\int_a^b |\psi(x)| dx $.
\end{thm}
Similar results apply for continuous functions and for the derivatives of differentiable functions \cite{attali1997approximations,hornik1990universal,pinkus1999approximation}.\par
These results open the way to the use of ANN for the numerical solution of differential equations. The most frequent approach when using ANN for this purpose is the minimization of a cost function in order to compute the optimal parameters $\bm{\alpha},\bm{\beta},\bm{w}$. This optimization procedure is referred in ANN as training. The optimization method can be also exploited for the resolution of differential problems because for many of these, the exact jacobian can be computed easily, see e. g. \cite{lagaris1998artificial}. Nevertheless, two main questions raise when training an ANN via optimization: the initial choice of the parameters and the overall computational cost. When dealing with time-dependent problems, the initial choice is related to the previous time step (see e.g. \cite{rudd2015constrained}); in other problems where the main difficulties are related to the geometry of the boundary, the initial choice is efficiently chosen as a boundary lift (see e.g. \cite{berg2018unified,lagaris1998artificial}). In the problem we consider here, there is no such a good indication by the problem itself that may help  to choose ``good" initial guesses. Here, we present a procedure that computes a solution in the framework of ELM networks that do not need training unless a big  accuracy is needed; in the later case it can be used for providing an initial guess for training purposes.
Toward to the above aim, we first show how one can set  the internal weights and biases to get good approximating functions.\par
Consider the function $G$ in \eqref{eq:ANNfun} that depends non-linearly on the internal weights and biases. Let us now take the function: 
\begin{equation}
\sigma_i(x)=\sigma(\alpha_i x + \beta_i)= \frac{1}{1+\text{exp}(-\alpha_i x-\beta_i)}\ .
\label{def:sigmai}
\end{equation}
The first and second derivatives of the above with respect to the independent variable $x$ read:
\begin{equation}
\begin{split}
\sigma_i'(x)=&\alpha_i\frac{\text{exp}(z)}{(1+\text{exp}(z))^2}\\
\sigma_i''(x)=&\alpha_i^2\frac{\text{exp}(z)(1-\text{exp}(z))}{(1+\text{exp}(z))^3},
\end{split}
\label{eq:dersigma}
\end{equation}
where $z=\alpha_i x + \beta_i$. Note, that if one takes two functions $\sigma_i$ where at least one of the parameters $\alpha,\beta$ is different, then these are linearly independent (see \cite{ito1996nonlinearity}).
Finally, the ratio between $\alpha_i$ and $\beta_i$ gives the location of the inflection point, while the internal weights $\alpha_i$ govern the variation of the amplitude of the \textit{S--shape}:
\begin{itemize}
\item $\sigma_i$ has an inflection point at $x=-\dfrac{\beta_i}{\alpha_i}$, that we call the \textit{center} $C_i$ of the sigmoid function;
\item $\sigma_i$ is monotone, $\lim_{x\to -\infty}=0; \lim_{x\to +\infty}=1$ if $\alpha_i$ is positive, the other way if negative. Moreover the range where the values are between 0.05 and 0.95 is $ \left[C_i -\frac{2.945}{\alpha_i},C_i+ \frac{2.945}{\alpha_i}\right]$. We denote this interval as the \textit{S--shape amplitude} $I_S$ of the function.
\end{itemize}
\begin{figure}[th]
\centering
\includegraphics[width=0.45\textwidth]{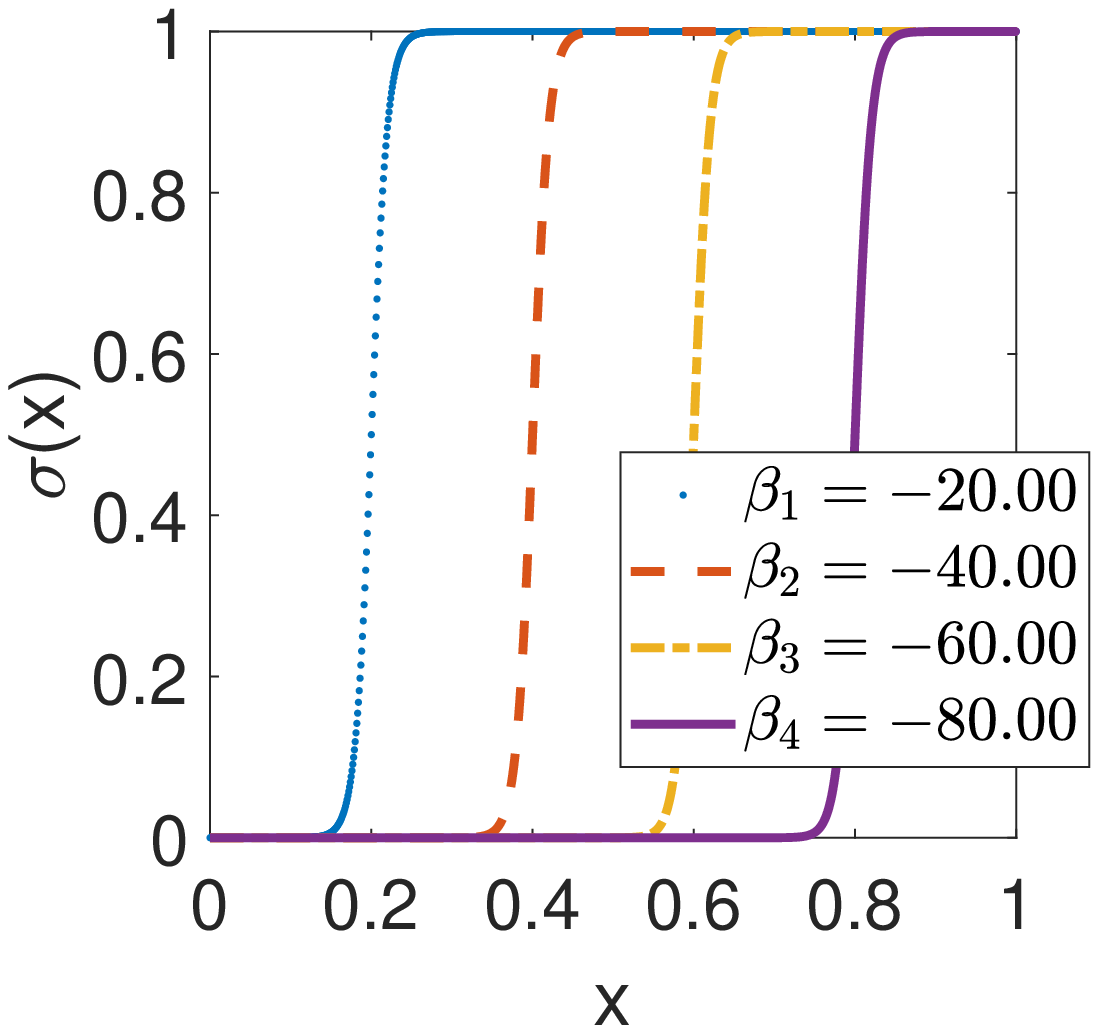} \, \includegraphics[width=0.45\textwidth]{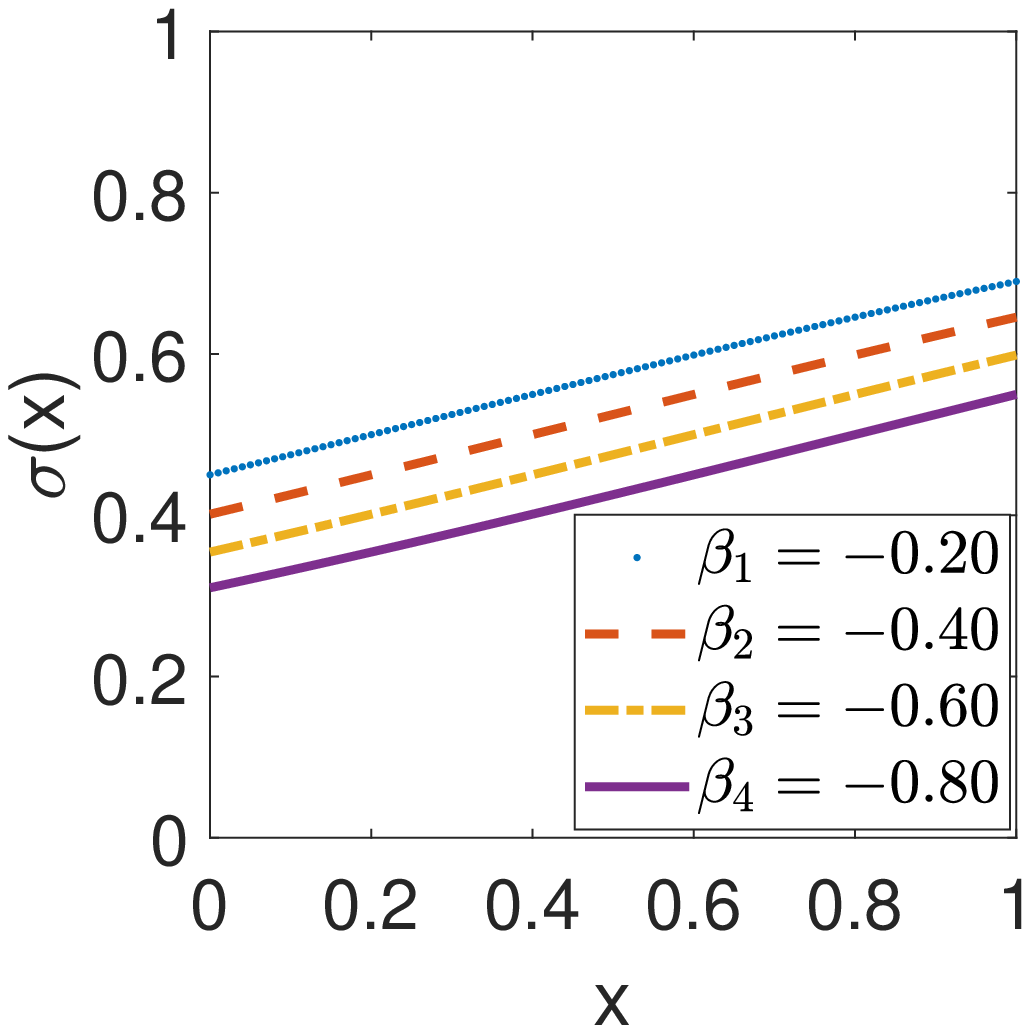}\\
\includegraphics[width=0.45\textwidth]{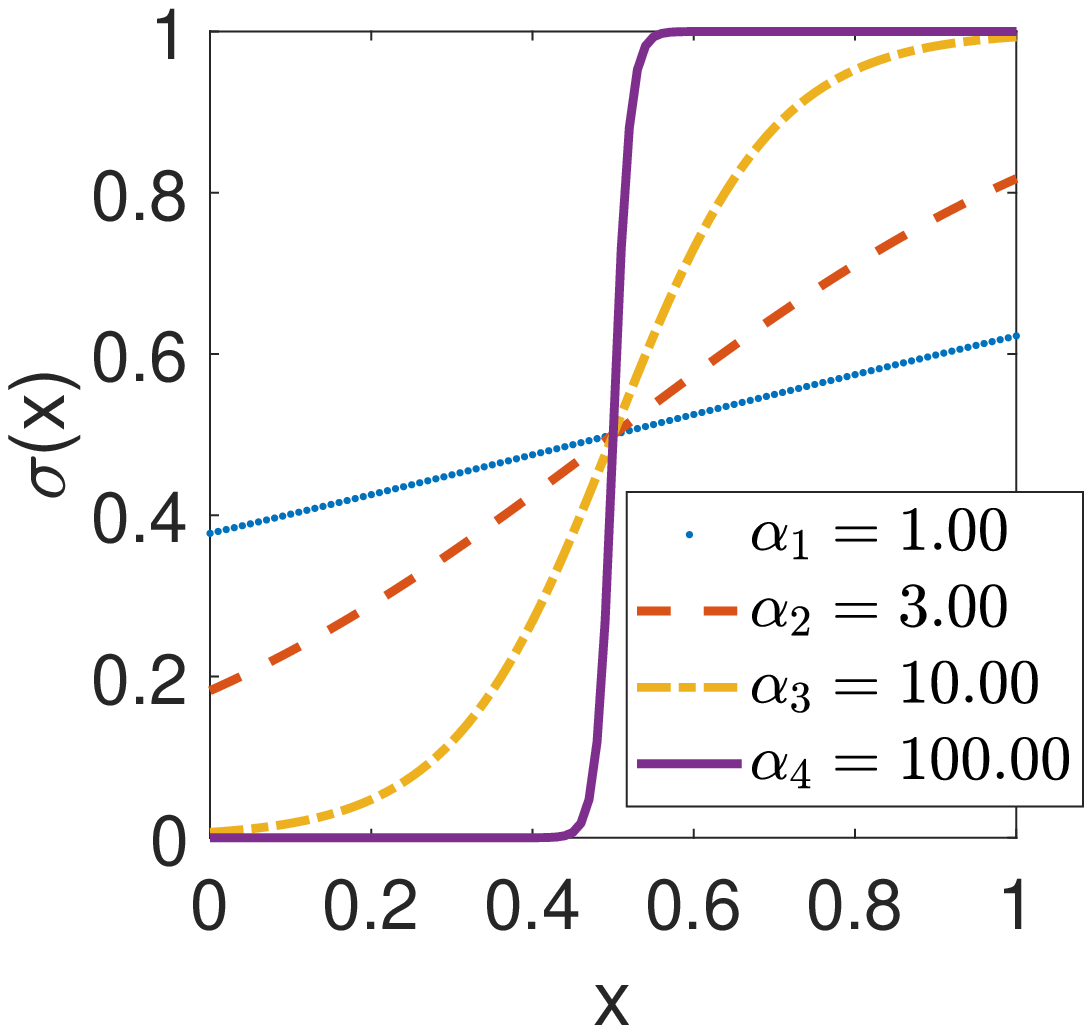}\, \includegraphics[width=0.45\textwidth]{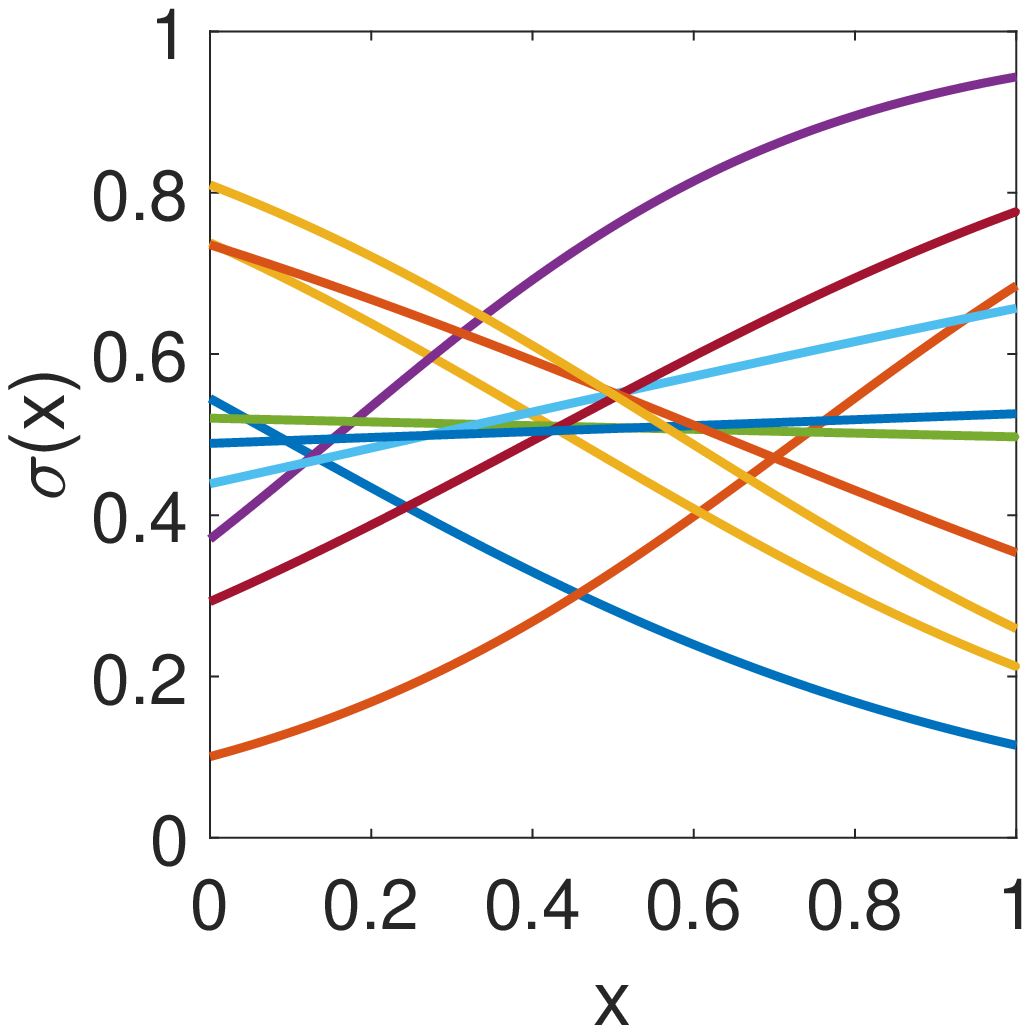}
\caption{The functions $\sigma_i$ of \eqref{def:sigmai} with varying parameters. On the top left panel, we have set $\alpha_i=100$, on the top right, $\alpha_i=1$. On the bottom left functions with a fixed center $C_i=1/2$ are depicted. On the bottom right, a set of 10 functions obtained by varying the coefficients as stated in Section \ref{sect:Model}.}
\label{Fig:fig1}
\end{figure}

If $\alpha_i,\beta_i $ are fixed once randomly, we keep their values constant, without updating them by tuning. This is the case of a ANN network referred as Extreme Learning Machine (ELM) network \cite{huang2011extreme,huang2004extreme,huang2006extreme}. The only free parameters need to be learned are the connections (weights) between the hidden layer and the output layer. The analysis reviewed in \cite{HUANG201532} gives the conditions that are required for ensuring the interpolation property, thus the non-singularity of the collocation matrices; we report this result in Section \ref{Sect:Analysis}.\par 
Outside $I_S$, the function $\sigma_i$ is almost constant; then, we focus on the values of the parameters $\alpha_i,\beta_i$ such that the intersection of the S--shape amplitude $I_S$ and the domain of interest is non-null, thus leading to functions $\sigma_i$ with $C_i$ inside the interval. The choice of $\alpha_i$ is related to the behavior of the derivative of the function. When taking all internal weights $\alpha_i\equiv K$, one obtains Heaviside-like functions if $K$ is big (as used in \cite{hahm2004approximation}), or almost-linear functions if $K$ is small (as used in \cite{guliyev2018approximation,lin2013approximation}). In our case, the use of both kind of functions can help to approximate both steep gradients and global behaviors. In figure \ref{Fig:fig1}, we plot different cases of such functions: on the top panels $\alpha_i$ is a constant, on the bottom left panel $C_i$ is set equal to the mid-point of the interval and finally on the bottom right panel, we plot functions with random values of $\alpha_i,\beta_i$. Note that some functions can have similar shapes: interpolation properties are not changed but the problem can become ill-conditioned. 
\section{The proposed numerical algorithm}\label{sect:Model}
In this section, we propose a simple algorithm to get a numerical approximation of a linear differential problem with the aid of ELM networks with $n$ neurons in the hidden layer. The main idea is to fix the range of values of weights and biases, so that we construct a ``good" basis of $n$ sigmoid functions $\sigma_i$, and use a collocation method to fix the weights of the linear output activation function.\par
For our illustrations, we have chosen a stationary one-dimensional PDE with constant coefficients.
\begin{equation}
\left\{ \begin{array}{l}
-\mu u''(x)+\gamma u'(x) + \lambda u(x)=f(x), \qquad x \in I:= [0,1] \subseteq \mathbb{R}\\
\nu_0 u'(0) + \rho_0 u(0)=g_0\,,\ \nu_1 u'(1) + \rho_1 u(1)=g_1,
\label{eq:ProDiff}
\end{array} \right.
\end{equation}
where $\mu$, $\gamma$ and $\lambda$  are respectively the diffusion, advection and reaction terms. The domain is fixed to $[0,1]$ for convenience. The boundary conditions are in the general Robin form; Dirichlet  and Neumann boundary conditions are derived by setting  $\nu_i=0$  or $\rho_i=0$.

\begin{algorithm}[th]
\SetAlgoLined
\textbf{Input: } coefficients $\mu, \gamma, \lambda; \nu, \rho $ \ ;\ number of neurons $n$, number of collocation points $M$, function evaluations $[\bm{f} , \bm{g}] \in \mathbb{R}^M$\\
 1. set $\bm{\alpha} \in \mathbb{R}^n$ randomly according to \eqref{eq:alpha} \;
 2. set $\bm{\beta}\in \mathbb{R}^n$ such that $C_i$ are uniformly distributed, i.e. $\beta_i=-\alpha_i*C_i$\;
 3. compute matrix $\mathbb{S}$ according to \eqref{eq:stiff}-\eqref{eq:M012} \;
 4. compute $\mathbb{B}$ according to \eqref{eq:boundary}-\eqref{eq:B01} \;
 5. solve the (eventually under/over -determined) linear system $ [ \mathbb{S};\mathbb{B}] \bm{w}= [\bm{f};\bm{g}]$ \;
\textbf{Output: }weights $\bm{w} \in \mathbb{R}^n$ that fix the ELM approximate solution $\tilde{u}(x)$ in \eqref{eq:ansatz};
 \caption{Pseudo-code that describe the ELM collocation method}
\end{algorithm}

We consider an approximate solution $\tilde{u}(x,\bm{w})$ given by the ELM:
\begin{equation}
\tilde{u}(x; \bm{w})=\sum_{i=1}^n w_i \sigma_i(x).
\label{eq:ansatz}
\end{equation}
We choose $\bm{\alpha}$ to be  random uniformly distributed in a range that depends on the number of neurons $n$; if we fix the domain of the differential problem to have unitary length, we choose:
\begin{equation}\label{eq:alpha}
\bm{\alpha} := {\text{rand}\left(\left[-\frac{n-10}{10}-4,\frac{n-10}{10}+4\right]\right)} \ .
\end{equation}
In this way, if $n=10$ we have a maximum weight e.g. $4$ and as $n$ grows by 10 then the maximum weight grows by $1$. The biases $\bm{\beta}$ are fixed so that the centers $C_i=-\beta_i/\alpha_i$ of the sigmoid functions are located on random uniformly distributed points inside the interval.\par
Then, the ELM network is collocated in \eqref{eq:ProDiff} on $M-2$ equidistant collocation points $x_j \in (0,1)$ getting:
\begin{equation}
-\mu \tilde{u}''(x_j)+\gamma \tilde{u}'(x_j) + \lambda \tilde{u}(x_j)=f(x_j) \qquad \text{for } j=2,\dots,M-1.
\label{eq:DTRcollocation}
\end{equation}
Analytical derivatives of the sigmoid function $\sigma_i$ can be computed by \eqref{eq:dersigma}. Then:
\begin{equation*}
-\mu \sum_{i=1}^n w_i \sigma_i''(x_j)+\gamma \sum_{i=1}^n w_i \sigma_i'(x_j) + \lambda \sum_{i=1}^n w_i \sigma_i(x_j)=f(x_j) \ \text{for } j=2,\dots,M-1.
\end{equation*}
Finally, we observe that we can write separately the external weights $\bm{w}$ and write this in a matrix form as:
\begin{equation}
\mathbb{S}\bm{w} :=(-\mu \mathbb{S}^{(2)}+\gamma \mathbb{S}^{(1)} + \lambda \mathbb{S}^{(0)})\bm{w} = \bm{f},
\label{eq:stiff}
\end{equation}
where $\mathbb{S}_2$, $\mathbb{S}_1$ and $\mathbb{S}_0$ are matrices of elements:
\begin{equation}
\begin{split}
\mathbb{S}^{(2)}=\bigg(\sigma_i''(x_j)\bigg)_{i,j} \ , \ \mathbb{S}^{(1)}=\bigg(\sigma_i'(x_j)\bigg)_{i,j} \ \text{and} \ \mathbb{S}^{(0)}=\bigg(\sigma_i(x_j)\bigg)_{i,j} \\
i=1,\dots, n \,,\ j=2,\dots, M-1
\end{split}
\label{eq:M012}
\end{equation}
For the boundary conditions defined in \eqref{eq:ProDiff}, we augment the system by two equations collocating on the boundary points $0$ and $1$
\begin{equation*}
\nu \sum_{i=1}^n w_i \sigma_i'(x_{k}) + \rho \sum_{i=1}^n w_i \sigma_i(x_{k})=g(x_{k}) \qquad
 k=0,M \,;\ x_1=0, x_M=1\ .
\label{eq:DNRcollocation}
\end{equation*}
As before, we rewrite the above in a matrix form:
\begin{equation}
\mathbb{B} \bm{w}:= (\nu \mathbb{B}^{(1)}+\rho \mathbb{B}^{(0)})\bm{w}  = \bm{g} \,,
\label{eq:boundary}
\end{equation}

where $\mathbb{B}^{(1)}$ and $\mathbb{B}^{(0)}$ are matrices of elements:
\begin{equation}
\begin{split}
\mathbb{B}^{(1)}=\bigg(\sigma_i'(x_{k})\bigg)_{i,{k}} \ \text{and} \ \mathbb{B}^{(0)}=\bigg(\sigma_i(x_{k})\bigg)_{i,{k}} \\
i=1,\dots, n \,,\  k=0,M \ .
\end{split}
\label{eq:B01}
\end{equation}
To this end, we have a linear system of $M$ equations and $n$ unknowns. The overall system can be constructed to be over-determined ($M>n$), under-determined ($M<n$) or squared ($M=n$). 
In the over- and under-determined cases, we consider the solution to be the one obtained in the least-square sense (in the under-determined problem the minimum-norm regularized solution). In the next section, we show that the under-determined solutions reach the same accuracy when compared with squared or over-determined cases. Moreover, we observe that the squared case can lead to ill-conditioned linear problems, so we propose under-determined collocation. 

\section{Analysis of the collocation approximation}\label{Sect:Analysis}
In this section, building on previous works \cite{auricchio2012isogeometric,hu2006collocation,huang2006extreme}, we prove the consistency of the proposed ELM collocation method. First, based on the Theorem 2.1 presented in \cite{huang2006extreme}, we state the following Theorem that fits to our proposed framework. 
\begin{thm}\label{Thm1}
Let $(x_i,y_i)\,,\ i=1,\dots, M $ be a set of points such that $x_{i}<x_{i+1}$, and take the ELM network with $n<M$ neurons $\tilde{u}(x; \bm{w})$ in \eqref{eq:ansatz} such that the internal weights $\bm{\alpha}$ and the biases $\bm{\beta}$ are randomly generated independently from the data according to any continuous probability distribution. Then, $\forall \varepsilon >0$, there exists a choice of $\bm{w}$ such that $\| (\tilde{u}(x_i; \bm{w}) - y_i)_i\| <\varepsilon $ \B. Here $\| \cdot\| $ denotes the $L^2$ Euclidean norm of vectors, the analogous of the Frobenius norm.\par
Moreover, if $n=M$ then  $\bm{w}$  can be found such that  $\| (\tilde{u}(x_i; \bm{w}) - y_i)_i \| =0 $.
\end{thm}
The above theorem states that if the number of hidden neurons are at least equal to the number of data points, then the interpolation error is zero.
Starting from the above theorem, we consider the properties of ELM network for the approximation of functions. To do this, we consider the construction of a sequence of nested ELM families: $ \mathbb{E}^{(n)}:= \{\tilde{u}(x; \bm{w})\,,\ \bm{w} \in \mathbb{R}^n \} \subset \mathbb{E}^{(n+1)}:= \{\tilde{u}(x; \bm{w})\,,\ \bm{w} \in \mathbb{R}^{n+1} \}$,  where we add one neuron at time.\par

Based on the above theorem, we can now state the following theorem (see e.g. Theorem 2 of \cite{HUANG201532}). 
\begin{thm}  \label{cor1} 
Let $\phi(x)$ be a continuous function. Then there exist a sequence of ELM network functions $\tilde{u}^{(n)}\in \mathbb{E}^{(n)} $ such that: 
\begin{equation*}
 \| \phi - \tilde{u}^{(n)} \|_{L^2} \to  0    \,, 
\end{equation*}
where  by $\|\cdot\|_2$ we denote the $L^2$ norm. 
\end{thm}

To prove the above, we consider a family of nested sets of ordered distinct points $\{x^{(M)}_i \,,\ i=1,\dots,M\} \subset \{x^{(M+1)}_i \,,\ i=1,\dots,M+1\}$ and call diameter of the set of points $H>0$ the maximum $H=max_{i=1, \dots, M-1}(x^{(M)}_{i+1}-x^{(M)}_{i})$. Then take in the space $\mathbb{E}^{(n)}$ and call $\tilde{u}^{(n)}$ the ELM network  that interpolates $\phi$ on a set of points of diameter $H$: from Theorem \ref{Thm1} this construction exists. Finally, we can say that  $\| (\phi(x_i) - \tilde{u}^{(n)}_H (x_i) )\| \to _{ H\to 0 } 0 $, using the vector norm: theorem \ref{cor1} extends the convergence to a general continuous function $\phi$ in the $L^2$ sense, thus states that the finite-dimensional discrete space of ELM networks is consistent in the $L^2$ setting.\par
Theorem \ref{Thm1} and \ref{cor1} also imply that the collocation matrices $\mathbb{S}^{(0)}$ are of full rank.\par 
We should note that for the problems that we consider here, in some cases we obtain ill-conditioned matrix problems. In these cases the solution of the under-determined collocation (number of neurons greater then number of points) with least squares is more suitable.\par
The collocation solution can now be sought as a solution of a variational formulation tested against delta-type functions or, equivalently, a discretized formulation of the problem with $C^2$ test functions. We follow the classic approach, see e.g. \cite{auricchio2012isogeometric,strang1973analysis} and write the weak formulation for the equation \eqref{eq:ProDiff} that seeks a $u(x)$ in the trial space (denoted by $U$) such that:
\begin{equation*}
\int_{0}^1\left[ -\mu u^{\prime \prime} (x) + \gamma u^\prime (x) + \lambda u (x) \right]\, v(x) \ dx = \int_{0}^1 f(x) v(x) \ dx
\end{equation*}
for all test function $v\in V$. For simplicity, we suppose that the boundary conditions are imposed on the space $U$ and are satisfied in all the subspaces introduced.\par
The general Galerkin method is a weak formulation as above where we look for a function $\tilde{u}_n \in \mathbb{U}_n$ for which the equations hold true $\forall v_m \in \mathbb{V}_m$ and both $\mathbb{U}_n$ and $\mathbb{V}_m$ are finite-dimensional. In our case, we fix $\mathbb{U}_n =\mathbb{E}^{(n)}$, i.e. the ELM space with $n$ neurons, so that the degrees of freedom of the finite dimensional space are the weights $\bm{w} \in \mathbb{R}^{n}$.\par
Now, consider the test space to  $\mathbb{V}_m=span\{ \delta_{i,\varepsilon }(x) \,,\ i=1,\dots, m \}$ defined starting from a set of $m=n$ distinct points $\{ x_i \}$. 
We also consider function $f$ to be approximated by its interpolant in the space of ELM, i.e. $\tilde{f}\approx f\,,\ \tilde{f}\in \mathbb{E}^{(n)}\,,\ \tilde{f}(x_i)=f(x_i)$.
Then the Galerkin problem reads:
\begin{equation}
\begin{split}
\int_{0}^1\left[ -\mu \tilde{u}_{\varepsilon }^{\prime \prime} (x) + \gamma \tilde{u}_{\varepsilon }^\prime (x) + \lambda \tilde{u}_{\varepsilon } (x) \right]\, \delta_{x_i,\varepsilon }(x) \ dx = \int_{0}^1 \tilde{f}(x) \delta_{x_i,\varepsilon }(x) \ dx \\ \quad \forall i=1,\dots, M
\end{split}
\label{eq:Gal_2}
\end{equation}
where functions $\delta$ are defined by
\begin{gather*}
\delta_{x_i,\varepsilon }(x) = \phi( |x-x_i| /\varepsilon )\,,   \\ 
\phi(\zeta)=\left\{\begin{array}{ll}  \dfrac{exp\{1/(\zeta^2-1)\}}{ \int exp\{1/(\zeta^2-1)\} \, d\zeta }  & \zeta\in [0,1) \\ 0 & \zeta\in [1,+ \infty) \end{array} \right.  .
\end{gather*}
This Galerkin problem is related to the above collocation problem by the following theorem. 
\begin{thm}\label{Thm2}
Let $\tilde{u}_{n,\varepsilon } (x)$ be a solution to the Galerkin problem  \eqref{eq:Gal_2}. Then $\lim_{\varepsilon\to 0 } \tilde{u}_{n,\varepsilon } (x) = \tilde{u}_n (x)$ where $\tilde{u}_n (x)$ is the solution in the ELM space of the collocation problem \eqref{eq:DTRcollocation} on the points $\{x_i\}$:
\begin{equation}
-\mu \tilde{u}_n^{\prime \prime}(x_i)+\gamma \tilde{u}_n^{\prime}(x_i) + \lambda \tilde{u}_n (x_i)=\tilde{f}(x_i) \qquad \text{for } i=1,\dots,M.
\label{eq:coll_n}
\end{equation} 
\end{thm}
This result follows from convergence of the integrals to evaluations of function on sites, due to the convergence of $\delta$ functions to $\delta$ distributions; a proof of this can be found in  \cite{auricchio2012isogeometric}.\par
Based on the Theorem 2.2 presented in \cite{strang1973analysis}, we are now ready to prove the following convergence theorem for our proposed scheme.
\begin{thm}\label{Thm3}
Let  $\{x^{(n)}_i \,,\ i=1,\dots,n\}_{n\in\mathbb{N}}$ be an increasing family of collocation points such that the diameter $H$ tends to zero. Moreover take $\mathbb{E}^{(n)}$ a family of ELM.  where we assume that matrix $\mathbb{S}$ of equation  \eqref{eq:stiff} constructed on the collocation points is full rank and uniformly invertible. \\
Then, if $\tilde{u}_{n } \in \mathbb{E}^{(n)} $ is ELM network solution to the collocation problem \eqref{eq:coll_n} and $u(x)$ is the solution to equation \eqref{eq:ProDiff}, we have that:
\begin{equation*}
\lim_{n\to \infty} \| \tilde{u}_{n} - u\|_{L^2} = 0 
\end{equation*}
\end{thm}
\begin{proof}
Denote by $L$ the differential operator in equation \eqref{eq:ProDiff}, such that $ Lu=f$. Call $Q_{n}:\phi\in\mathbb{R}^n$ the operator that gives evaluation of function collocated on points $\{x_i\}$, i.e. $Q_{n}(\phi)= (\phi(x_i))_i$.  Then call $P_{n}:\phi\in\mathbb{R}^n$ the operator that gives the ELM network that interpolates functions at collocation nodes, i.e. $P_{n}(\phi)=\tilde{\phi}:=\sum w_j\sigma_j(x)$ such that $\tilde{\phi}(x_i)= \phi(x_i)$. Then the solution $\tilde{u}_{n}$ to the collocation problem can be seen as the solution to  $ Q_{n} L P_{n} \tilde{u}_{n}= Q_{n} f$, and matrix $\mathbb{S}$ of equation \eqref{eq:stiff} is the discrete representation of the operator $ Q_{n} L P_{n}$. With this we have:
\begin{gather*}
        Lu=f \Rightarrow 
        Q_{n}L u= Q_{n} f \Rightarrow \\
        Q_{n}L u - Q_{n}L P_{n} u + Q_{n}L P_{n} u = Q_{n} f\Rightarrow\\
        Q_{n}L P_{n} u + Q_{n}L (u-P_{n} u)=Q_{n} f\Rightarrow\\
        u - (Q_{n}L P_{n})^{-1} Q_{n} f = (Q_{n}L P_{n})^{-1} Q_{n}L (P_{n} u -u)
\end{gather*}
At last notice that from $ Q_{n} L P_{n} \tilde{u}_{n}= Q_{n} f$ we can write $\tilde{u}_{n}=  (Q_{n}L P_{n})^{-1} Q_{n} f$. Finally:
\begin{equation*}
\|u - \tilde{u}_{n}\| \le \|(Q_{n}L P_{n})^{-1} \| \, \|Q_{n}\|\, \|L\|\, \|u-P_{n} u\| \ .
\end{equation*}
In this, we have that the first term is bounded from the hypothesis on the operator $\mathbb{S}^{-1}$; the operator $Q_{n}$ is trivially bounded in our setting; the linearity of operator $L$ imply that also the third term is bounded. The thesis then follows from the convergence properties stated in theorems \ref{Thm1}-\ref{cor1}.
\end{proof}
We remark that the collocation equations can be derived also via quadrature rules applied to equations \eqref{eq:Gal_2}. In this case there is no need of Theorem \ref{Thm2}, and the relation with collocation is obtained by the requirement that the quadrature rule is supported on the collocation points. In this case, the functions $\delta$ can be substituted with any family of locally supported functions whose support includes only one collocation point.

\section{Numerical results}\label{Sect:Test}
In this section, we apply the proposed method for the solution of benchmark problems exhibiting steep gradients as suggested in \cite{MITCHELL2013350}. More particularly, in Section \ref{sec:regular}, we first consider regular problems, in Section \ref{sec:boundary} boundary layer problems, and in Section \ref{sec:internal} internal layer problems.
For our computations, we have used Matlab2020b. In order to get reproducible results, we have used rng(5) to create the random values for the vectors $\bm{\alpha}$, $\bm{\beta}$ before calling the random value generator of Matlab (rand). The solution of the least squares problems is achieved via the backslash command.\par 
For each problem, the number $n$ of neurons ranged from 10 to 1280, doubling the number of neurons at each execution. For each fixed $n$, we compute solutions with various numbers $M$ of collocation points.
The computational time for all the above problems was of the order of 0.5 sec on a single intel core i7-10750H with 16GB RAM 2.60Ghz. In particular the maximum computational time was ~0.6 sec for $n=1280$.
To evaluate the approximation error, we used two different metrics: the absolute difference between computed and exact solution, denoted by $E_u$, and the residual of the equation, denoted by $R_f$:
\begin{equation*}
\begin{split}
E_u(x)=&|u(x)-\tilde{u}(x)|\\
R_f(x)=&|f(x)+\mu\tilde{u}^{\prime \prime}(x)-\gamma\tilde{u}^\prime(x)-\lambda\tilde{u}(x)| \ .
\end{split}
\label{eq:err}
\end{equation*}
Note, that all errors are absolute, not normalized.
Then, we computed the $L^2$ norm of $E_u$ and $R_f$; when computing the $L^2$-norm, we used 5000 equispaced points and the trapezoidal rule for integration. As reference numerical method, we considered the finite difference (FD) scheme obtained with a maximal order on 7 nodes computed on $M$ equispaced points. The FD function considered here is the piecewise linear polynomial one interpolating the computed values. We emphasize that in all computations the number of points considered for the FD solution is equal to $M$, i.e. the number of collocation equations considered for the ELM implementation. In two cases, we also considered the solution only at collocation points, thus considering $L^\infty$ errors between vectors in $\mathbb{R}^M$.

\subsection{Regular problems}\label{sec:regular}
In this section, we analyze sinusoidal-bump problems as well as a high-order polynomial problem to show that the proposed approach can result to a high approximation accuracy.\par
The first example is a simple 1D boundary value problem with homogeneous Dirichlet boundary conditions:
\begin{equation}
\left\{ \begin{array}{l}
u''+(4k^2\pi^2-1)u=4k\pi e^x\cos(2k\pi x), \qquad 0<x<1\\
u(0)=0\,,\ u(1)=0 \qquad \text{ with } k \in \mathbb{N}
\label{eq:sin2kpi}
\end{array} \right.
\end{equation}
The above has the exact analytical solution: 
\begin{equation*}
u=\text{exp}(x)\sin(2 k \pi x) 
\end{equation*}
The coefficient $k$ represents the number of oscillations in the domain. We consider the simple case with $k=1$ and a more challenging one with $k=5$ (see Figure \ref{Test:reg1}. The computed errors with respect to the exact solution are reported in Figure \ref{Test:reg1} and in Table \ref{Tab:1}. In both cases, we note that the proposed ELM network results to a good approximation with a modest number of neurons. 


\begin{figure}[ht]
\centering
\includegraphics[width=0.45 \textwidth]{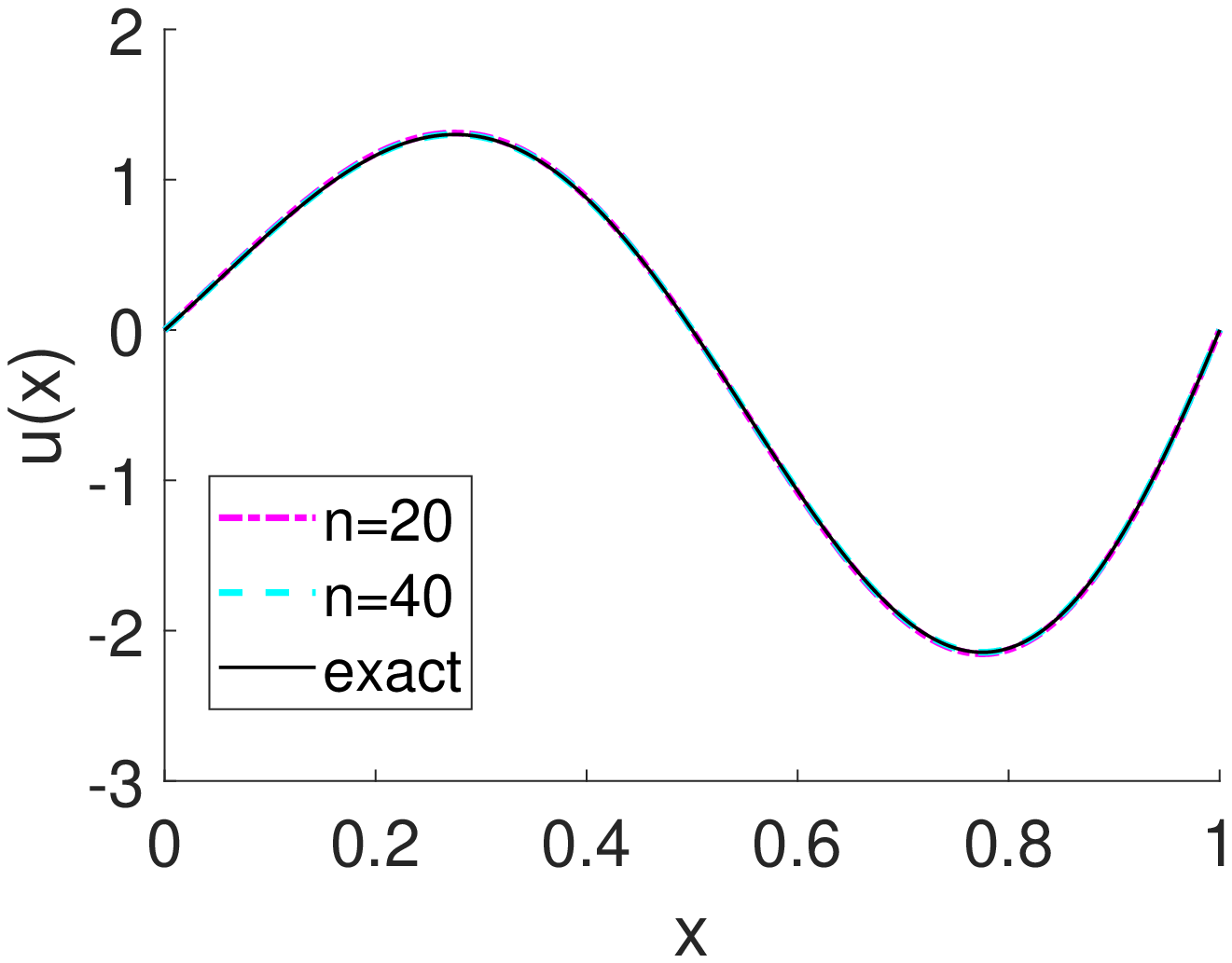}
\,
\includegraphics[width=0.45 \textwidth]{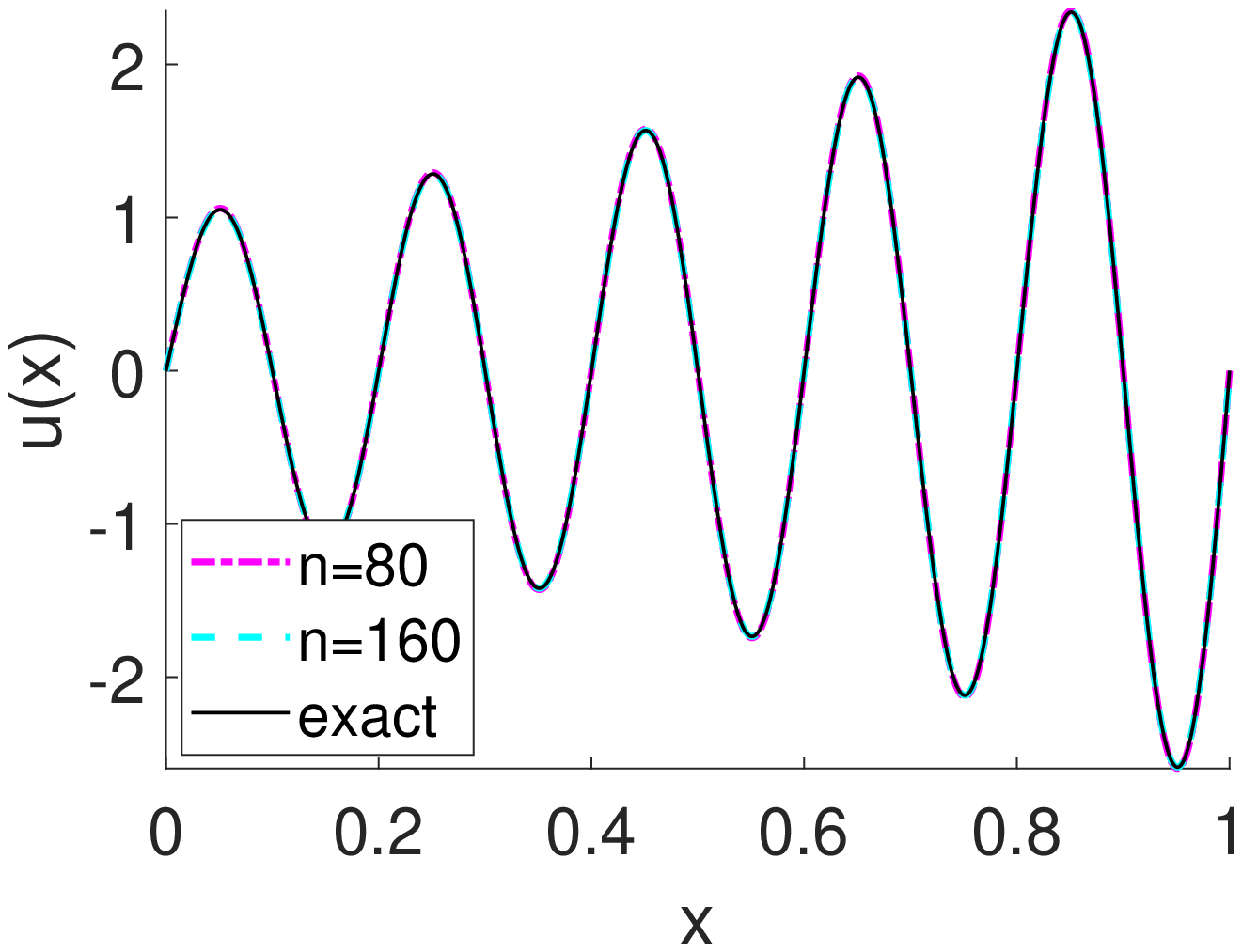}
\caption{Applying the proposed ELM network for the numerical solution of the problem given in \eqref{eq:sin2kpi}: approximation with $n$ neurons, the number of collocation points $M$ fixed to be $n/2$ (on the left panel, $k=1$, on the right panel, $k=5$.)}
\label{Test:reg1}
\end{figure}

\begin{figure}[th]
\centering
\includegraphics[width=0.45 \textwidth]{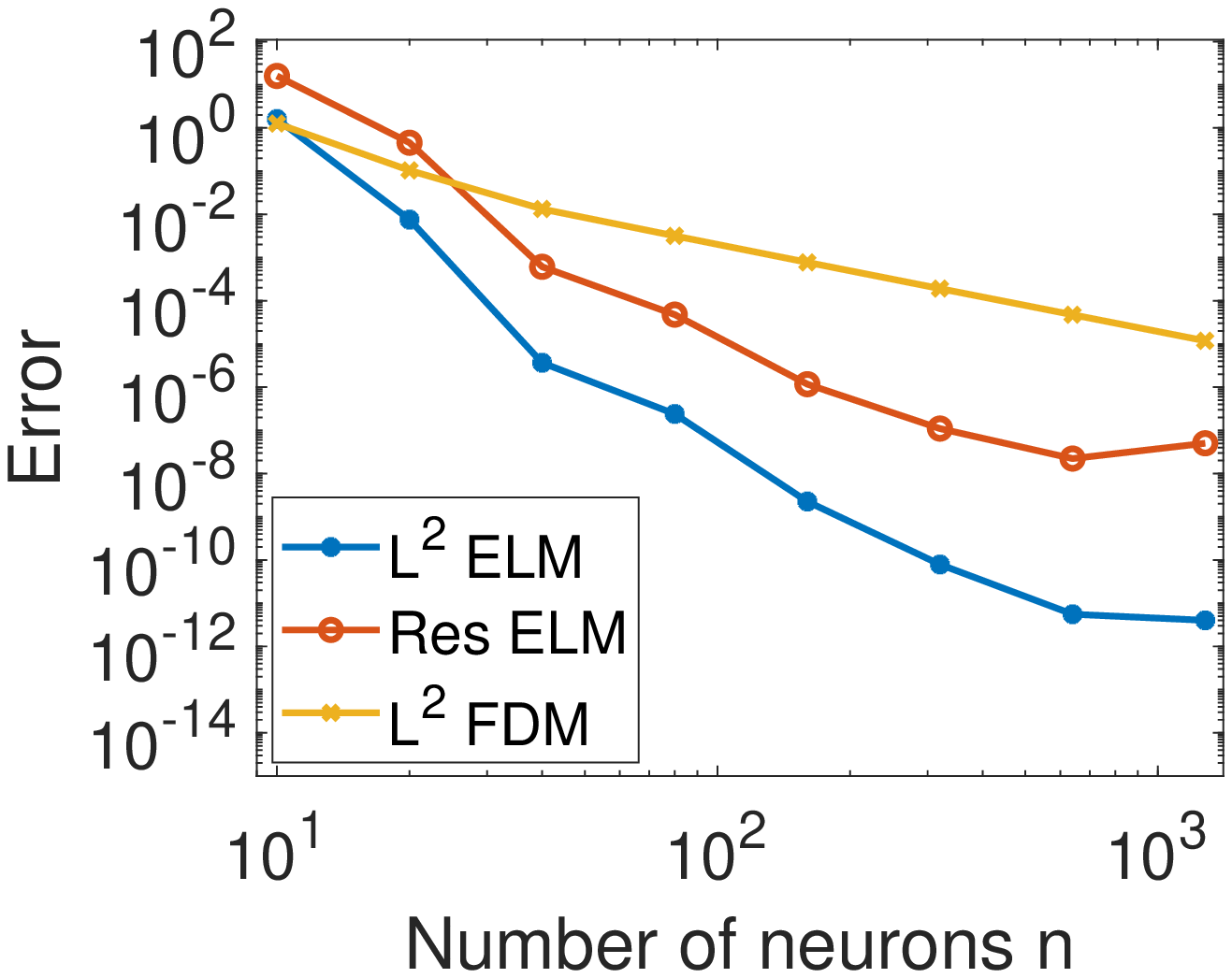}
\,
\includegraphics[width=0.45 \textwidth]{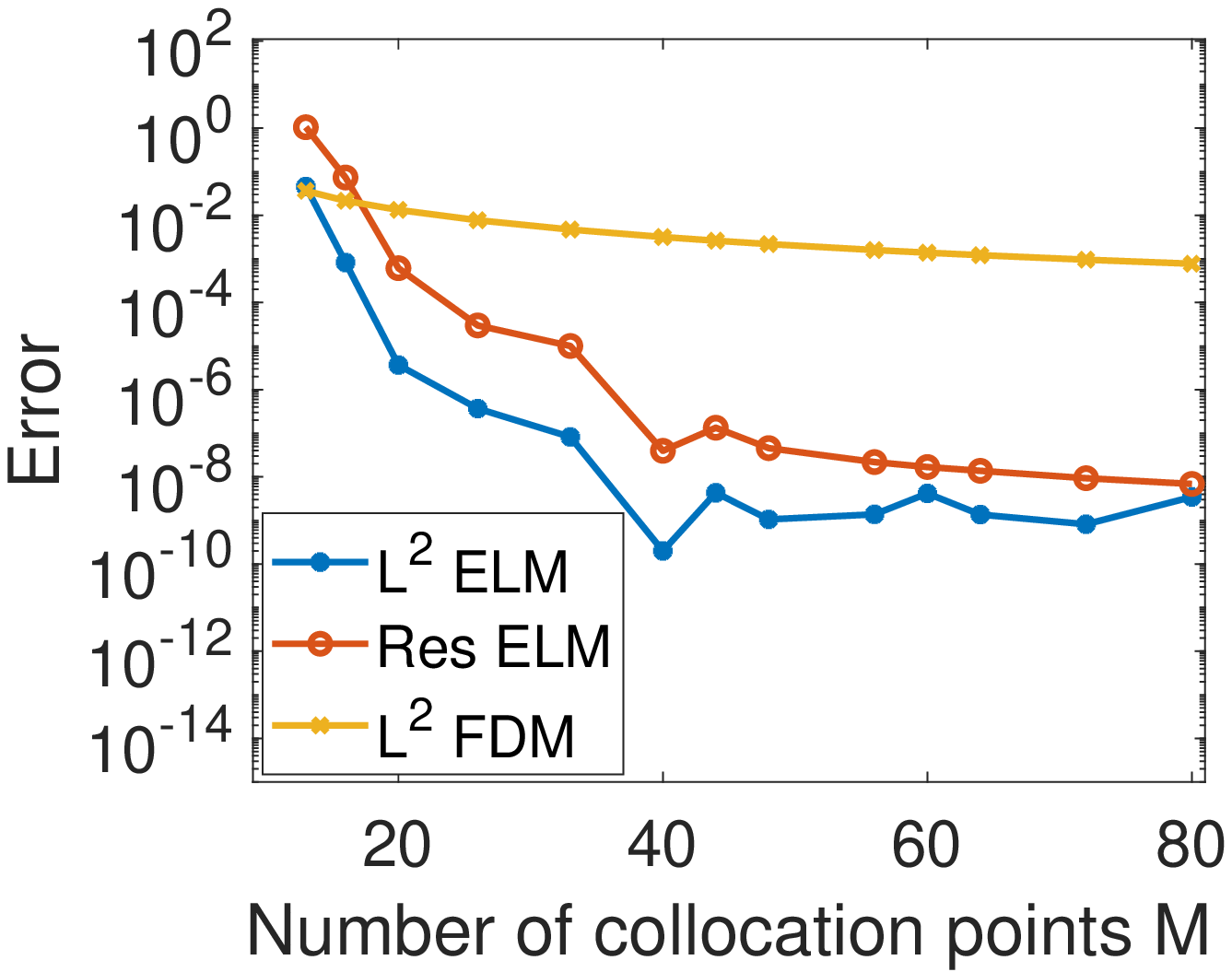}\\
\includegraphics[width=0.45 \textwidth]{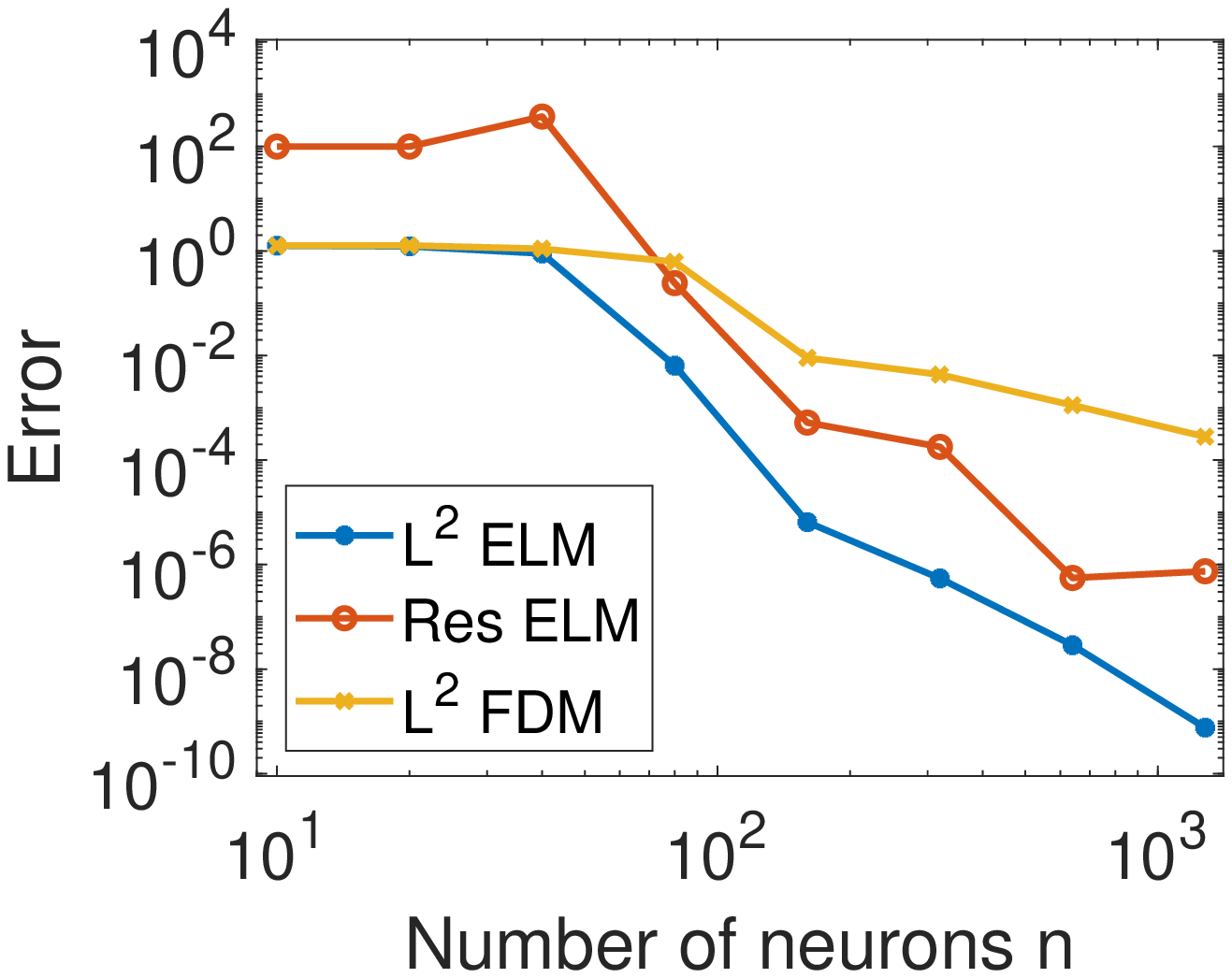}
\,
\includegraphics[width=0.45 \textwidth]{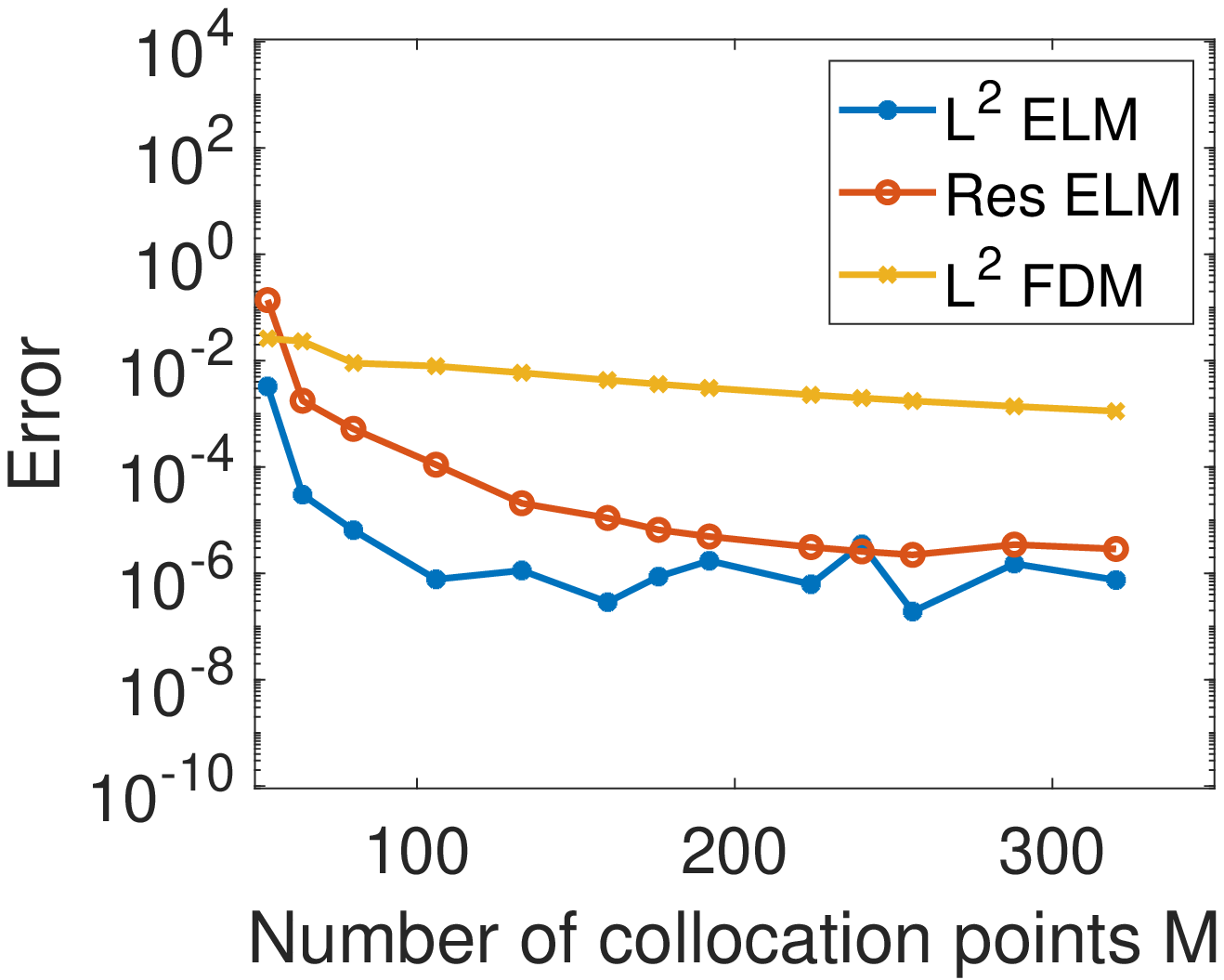}

\caption{Error and residual convergence for the ELM numerical solution of the problem \eqref{eq:sin2kpi}. On the top panels, $k=1$: on the left top panel $M=n/2$, on the right top panel $n=40$. On the bottom panels, $k=5$: on the left bottom panel, $M=n/2$, on the right bottom panel, $n=160$.}
\label{Test:reg1_errK1}
\end{figure}

\begin{figure}[ht]
\centering
\includegraphics[width=0.45 \textwidth]{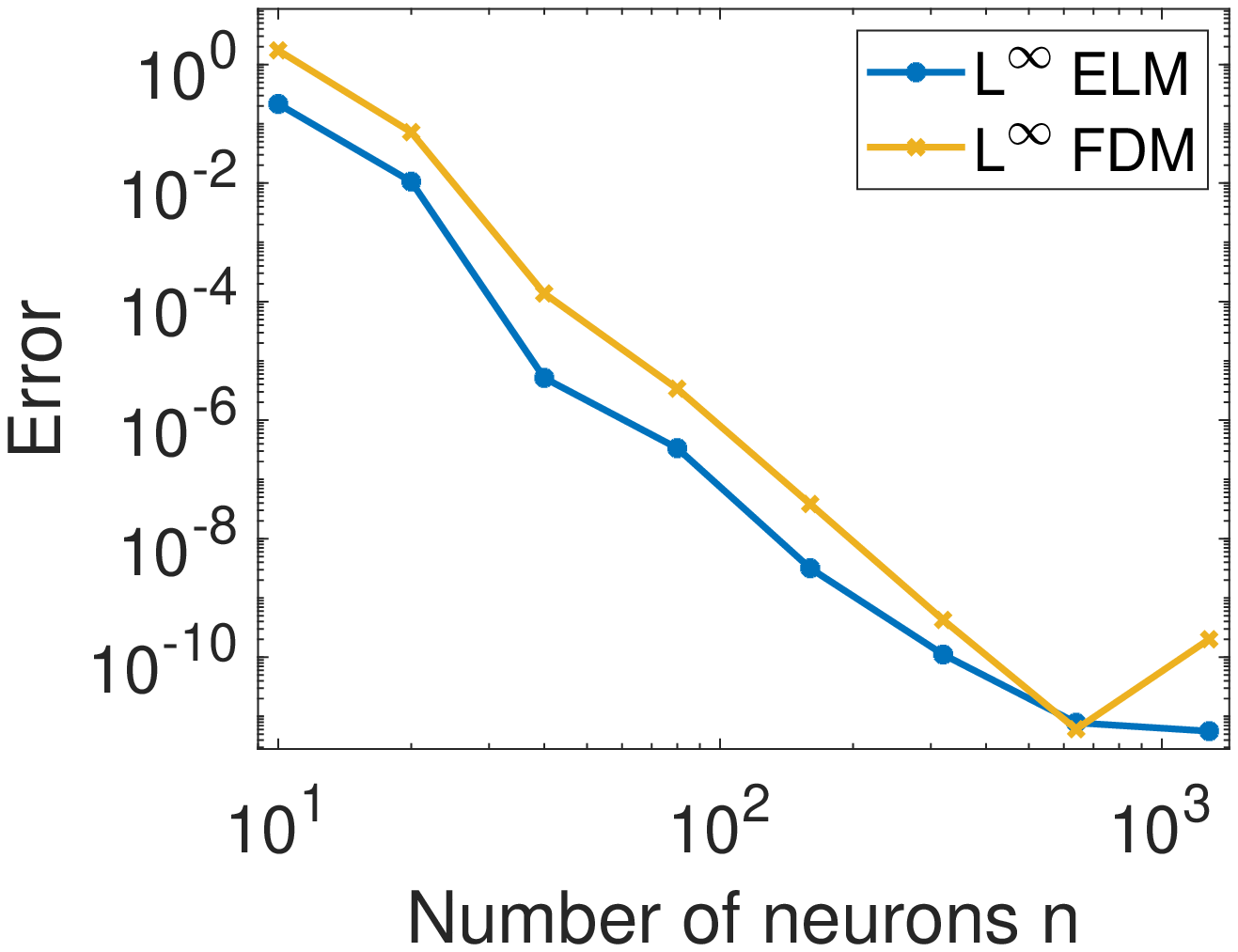}
\,
\includegraphics[width=0.45 \textwidth]{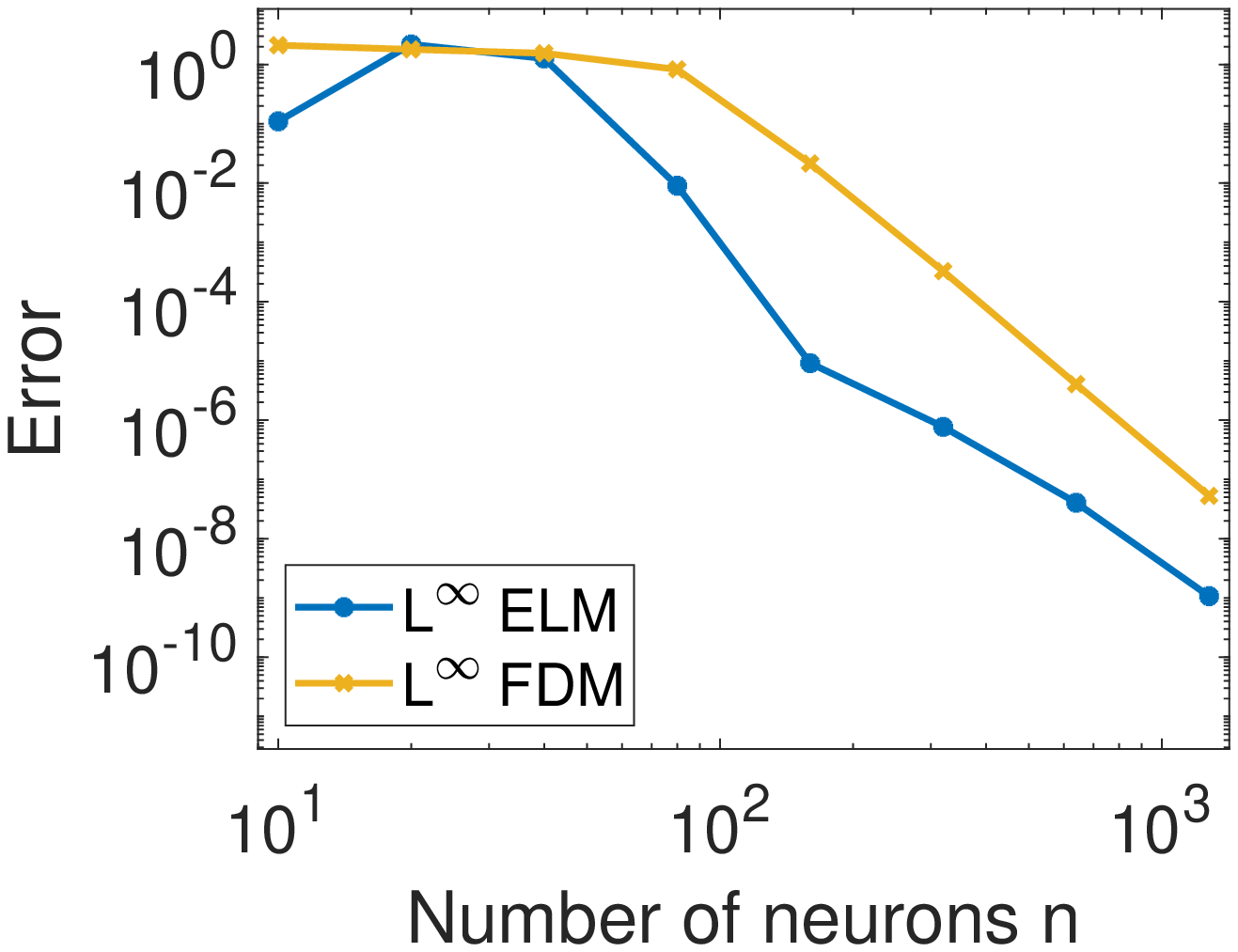}
\caption{Converge plots of the $L^\infty$ norm at the collocation points for the solution of the problem \eqref{eq:sin2kpi}: on the left panel, $k=1$, on the right panel, $k=5$.}
\label{Test:reg1_SuperConv}
\end{figure}

\begin{table}[th]
\tiny 
\begin{tabular}{|ll|ll|}
\hline
k=1 & & & \\
neurons & nodes & Error\_L2 & Residual \\ 
\hline \hline
10 & 3  & 1.2825e+00 & 1.7566e+01 \\
   & 4  & 1.1336e+00 & 1.4846e+01 \\
   & 5  & 1.6084e+00 & 1.5970e+01 \\
   & 6  & 6.4758e+00 & 4.7969e+01 \\
   & 8  & 1.9045e-01 & 6.7843e+00 \\
   & 10 & 1.1092e-02 & 1.5801e-01 \\ \hline
20 & 6  & 4.1912e-01 & 2.1314e+01 \\
   & 8  & 4.1948e-01 & 7.0691e+00 \\
   & 10 & 7.5525e-03 & 4.5484e-01 \\
   & 13 & 2.1932e-03 & 4.7934e-02 \\
   & 16 & 9.4571e-05 & 4.7279e-03 \\
   & 20 & 3.4250e-07 & 2.3290e-05 \\ \hline
40 & 13 & 4.5810e-02 & 1.0488e+00 \\
   & 16 & 8.2667e-04 & 7.3745e-02 \\
   & 20 & 3.6740e-06 & 6.1176e-04 \\
   & 26 & 3.6832e-07 & 2.9772e-05 \\
   & 33 & 8.2306e-08 & 1.0047e-05 \\
   & 40 & 1.9954e-10 & 3.9535e-08 \\ \hline
80 & 26 & 1.4968e-04 & 1.1856e-02 \\
   & 32 & 1.4893e-05 & 1.9670e-03 \\
   & 40 & 2.3953e-07 & 4.7799e-05 \\
   & 53 & 1.9385e-10 & 5.7391e-08 \\
   & 66 & 3.0444e-10 & 6.2147e-08 \\
   & 80 & 2.3971e-11 & 5.9213e-09 \\ \hline
\end{tabular}
\,
\begin{tabular}{|ll|ll|}
\hline 
k=5 & & & \\
neurons & nodes & Error\_L2 & Residual \\ 
\hline  \hline
40  & 13  & 1.0291e+00 & 5.4833e+02 \\
    & 16  & 9.0266e-01 & 2.8111e+02 \\
    & 20  & 8.9477e-01 & 3.7287e+02 \\
    & 26  & 1.0524e+00 & 2.7019e+01 \\
    & 33  & 1.1745e+00 & 1.1760e+00 \\
    & 40  & 3.0679e-01 & 8.3835e+00 \\ \hline
80  & 26  & 1.0283e+00 & 2.7422e+01 \\
    & 32  & 2.3444e-02 & 1.4636e+00 \\
    & 40  & 6.3788e-03 & 2.4238e-01 \\
    & 53  & 6.1162e-04 & 3.7449e-02 \\
    & 66  & 2.7654e-06 & 1.3766e-03 \\
    & 80  & 1.0742e-06 & 1.9191e-04 \\ \hline
160 & 53  & 3.2723e-03 & 1.3777e-01 \\
    & 64  & 3.0106e-05 & 1.7641e-03 \\
    & 80  & 6.4715e-06 & 5.2040e-04 \\
    & 106 & 7.7111e-07 & 1.1040e-04 \\
    & 133 & 1.1303e-06 & 2.0768e-05 \\
    & 160 & 2.8612e-07 & 1.0990e-05 \\ \hline
320 & 106 & 5.1808e-06 & 9.5524e-04 \\
    & 128 & 6.5518e-06 & 1.1001e-03 \\
    & 160 & 5.4403e-07 & 1.7879e-04 \\
    & 213 & 1.3397e-07 & 1.0850e-05 \\
    & 266 & 9.3901e-08 & 4.1553e-06 \\
    & 320 & 2.6074e-09 & 8.7342e-08 \\ \hline
\end{tabular} \caption{Error and residuals for the ELM numerical solution of problem \eqref{eq:sin2kpi}: on the left table, $k=1$, on the right table $k=5$.} \label{Tab:1}
\end{table}
\normalsize

The second boundary value problem that we consider here is a problem containing a high-order polynomial:
\begin{equation}
\left \{ \begin{array}{l}
u''=2^{2p}px(1-x)^{p-2}x^{p-2}(-1+2x-2x^2+p(1-4x+4x^2)), \ 0<x<1\\
u(0)=0\,,\ u(1)=0
\end{array} \right.
\label{eq:polinomiale}
\end{equation}
The above BV problem has the exact solution: 
\begin{equation}
u(x)=2^{2p}x^p(1-x)^p
\label{eq:sol_esatta_poly}
\end{equation}
Here, we consider the case $p=10$, so that to get a non-exact solution with FD.

\begin{figure}[th]
\centering
\includegraphics[width=0.45 \textwidth]{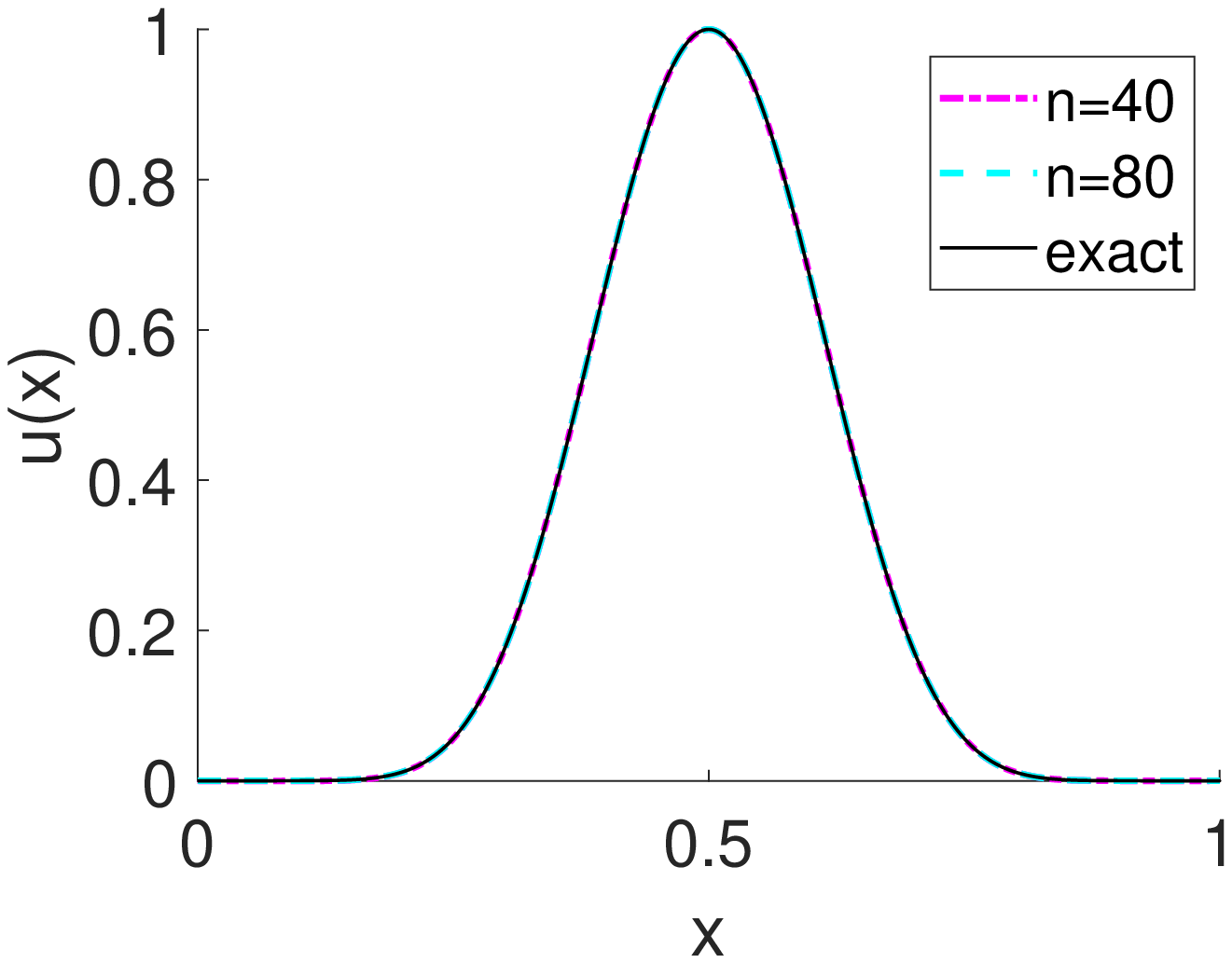}\\
\includegraphics[width=0.45 \textwidth]{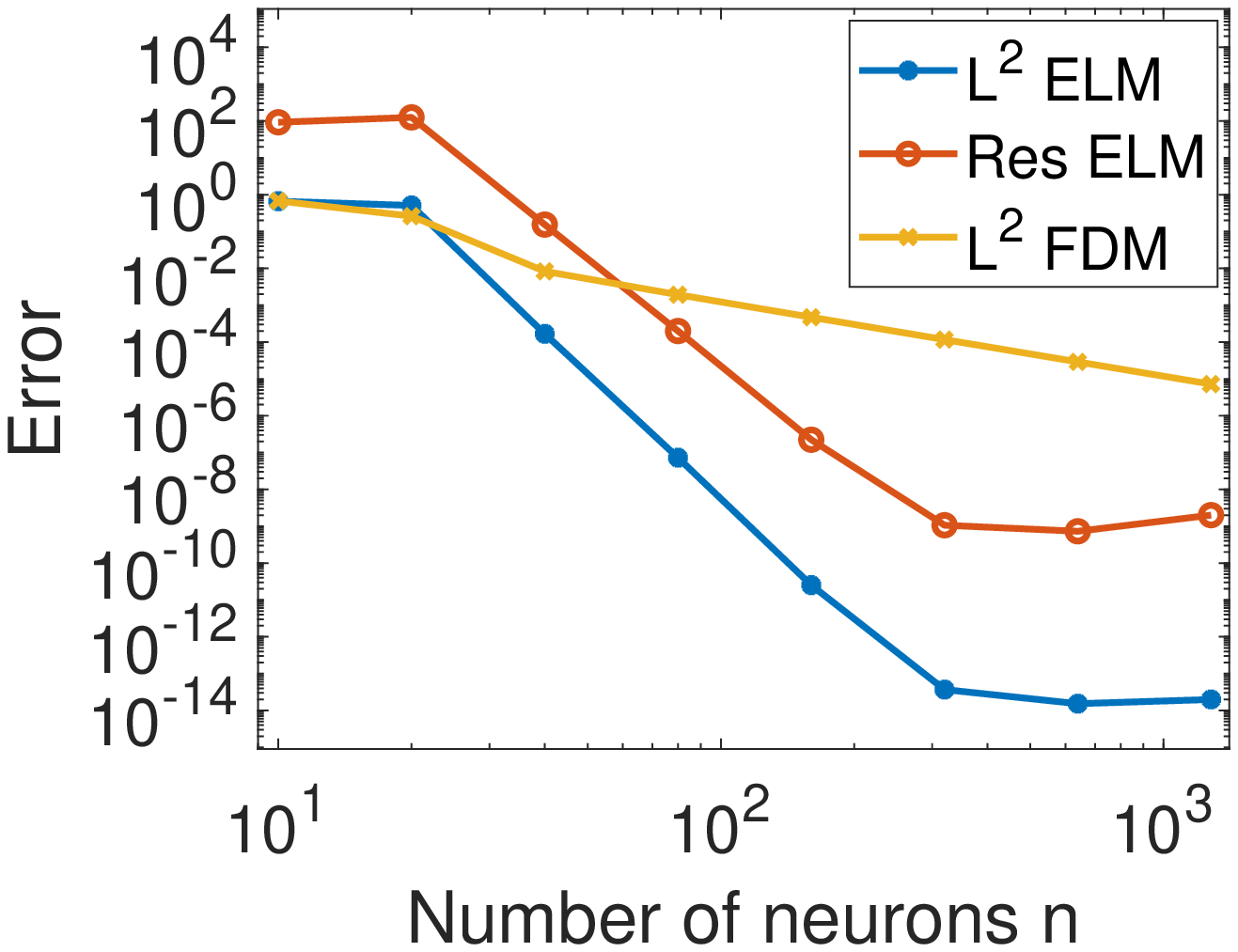}
\,
\includegraphics[width=0.45 \textwidth]{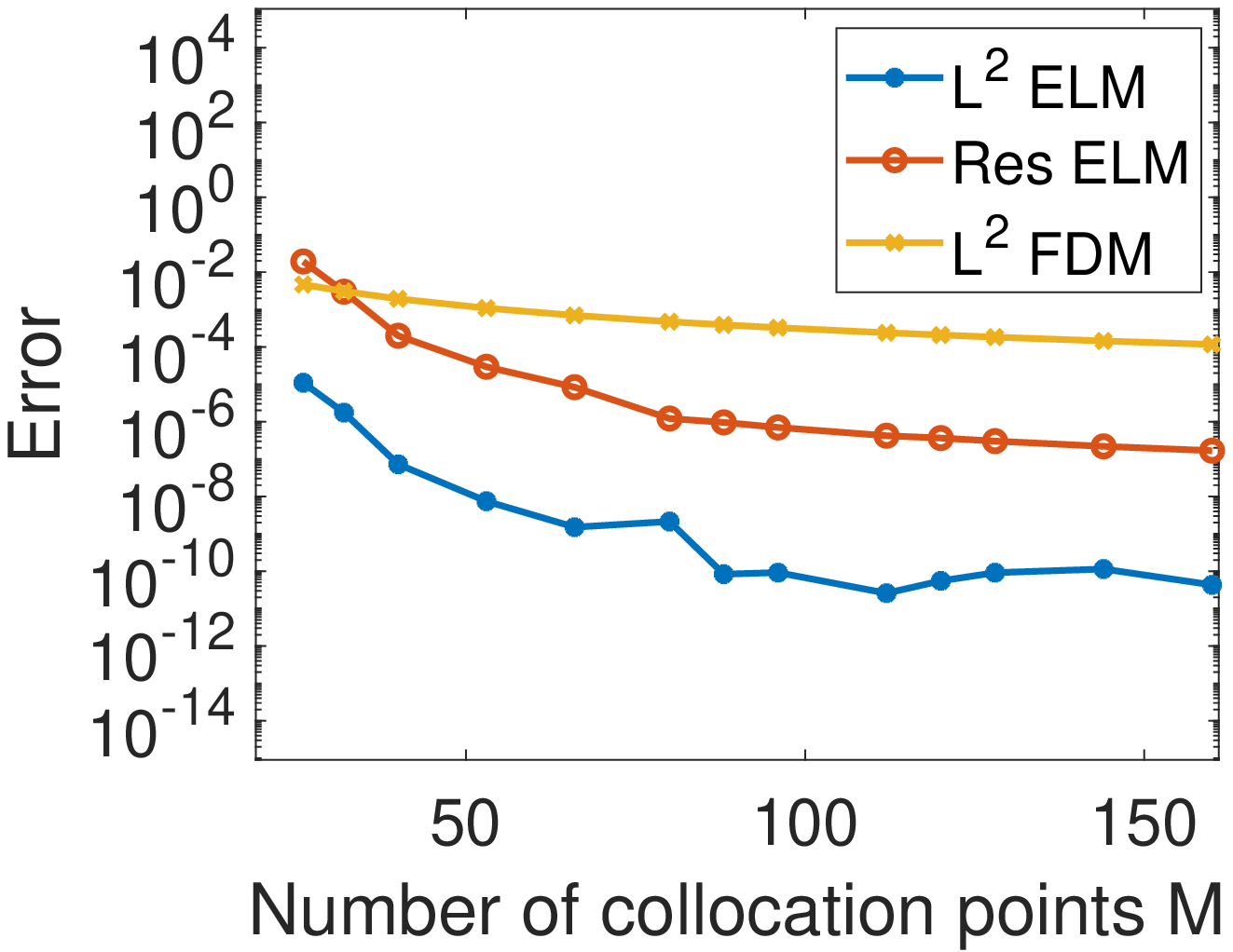}
\caption{Numerical tests for the solution of problem \eqref{eq:polinomiale} with $p=10$. On the top panel are depicted the exact \eqref{eq:sol_esatta_poly} and the ELM numerical solutions with $n=40,80$; the number of collocation points $M$ was set to $n/2$. On the bottom panels are shown the error and residual convergence: on the left panel, we fix $M=n/2$ and vary $n$ and on the right panel we fix $n=80$ and vary $M$}.
\label{Test:reg3_err}
\end{figure}

The computed errors arising from the proposed ELM network with respect to the exact analytical solution are reported in Figure \ref{Test:reg3_err}.

\subsection{Boundary Layer problems} \label{sec:boundary}
In this section, we consider benchmark boundary layer problems, namely an advection-dominated and a reaction-dominated problem that lead to sharp gradients near the boundary (see also \cite{quarteroniLibro}). We show how the proposed scheme can deal properly with the steep gradients that arise in these cases.\par
When $\lambda=0$ and $f\equiv0$ in \eqref{eq:ProDiff}, we have a diffusion-advection problem. If we impose Dirichlet boundary conditions with $g_0=0\,,\ g_1=1$, the solution reads:
\begin{equation}
u(x)=\dfrac{\text{exp}\left(\dfrac{\gamma}{\mu}x\right)-1}{\text{exp}\left(\dfrac{\gamma}{\mu}\right)-1} \ .
\label{eq:exact_1}
\end{equation}
If $|\frac{\gamma}{\mu}|<<1$, the solution approaches the line connecting the boundary conditions while
if $|\frac{\gamma}{\mu}|>>1$ 
the solution is near to zero in almost the whole domain, except in a neighborhood of the right boundary where the function has a steep gradient. In this last case, we say that the solution has a \textit{boundary layer} of width order $\mathcal{O}\left(\dfrac{\mu}{\gamma}\right)$. From a numerical point of view, for the particular problem, in which advection dominates diffusion, we need to catch the behavior on a small scale $\frac{\mu}{\gamma}$, but it is well known that using finite difference or finite element methods, the approximate solution can oscillate while the exact solution is monotone (see \cite{quarteroniLibro}).\par
In the advection-dominated problem,  the presence of the boundary layer is usually related to the so-called global P\'eclet number:
\begin{equation}
\mathbb{P}e_g=\frac{|\gamma|\cdot|I|}{2\mu}
\label{eq:Pe_1}
\end{equation}
where $|I|$ is the domain length.

\begin{figure}[th]
\centering
\includegraphics[width=0.45 \textwidth]{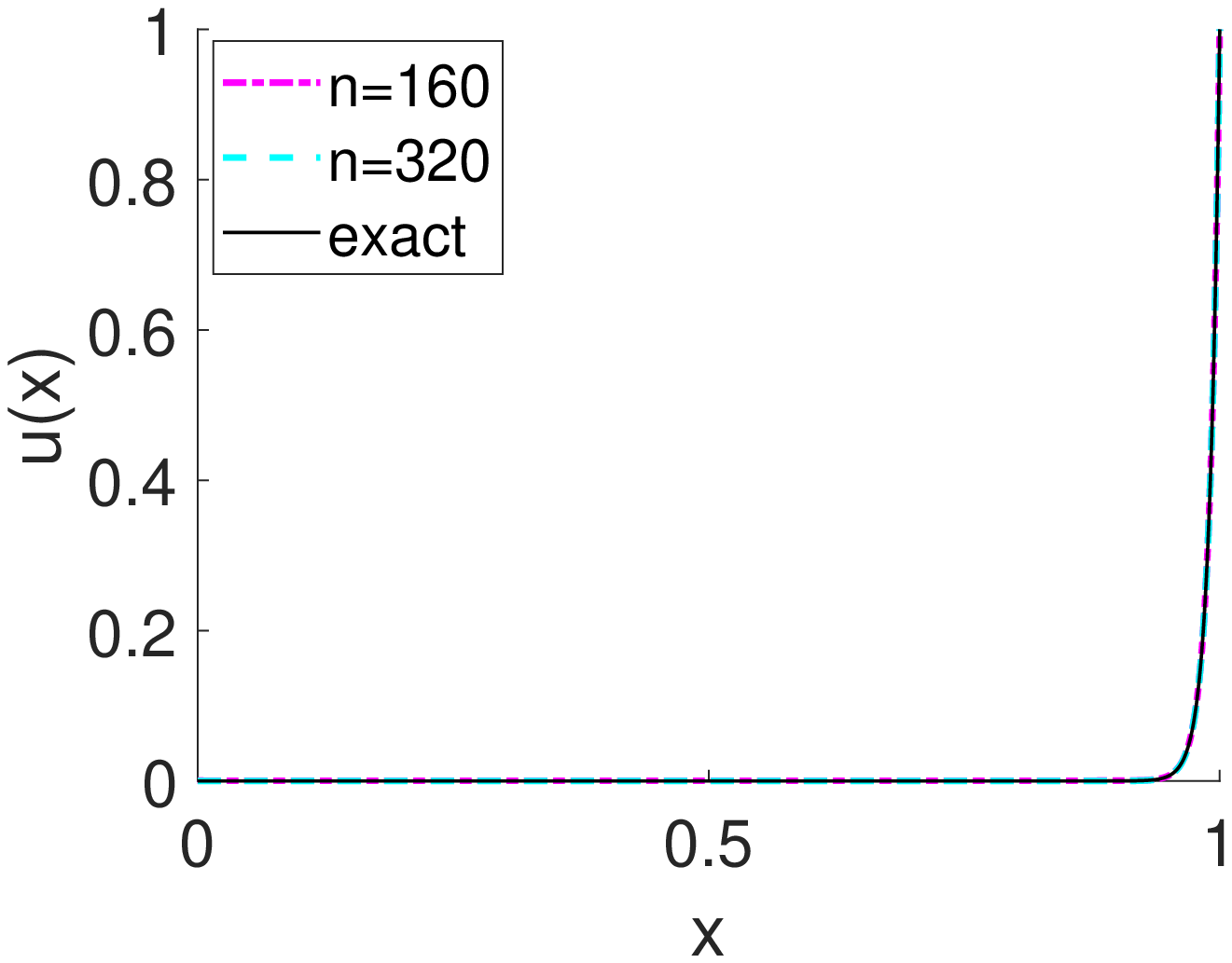}\\
\includegraphics[width=0.45 \textwidth]{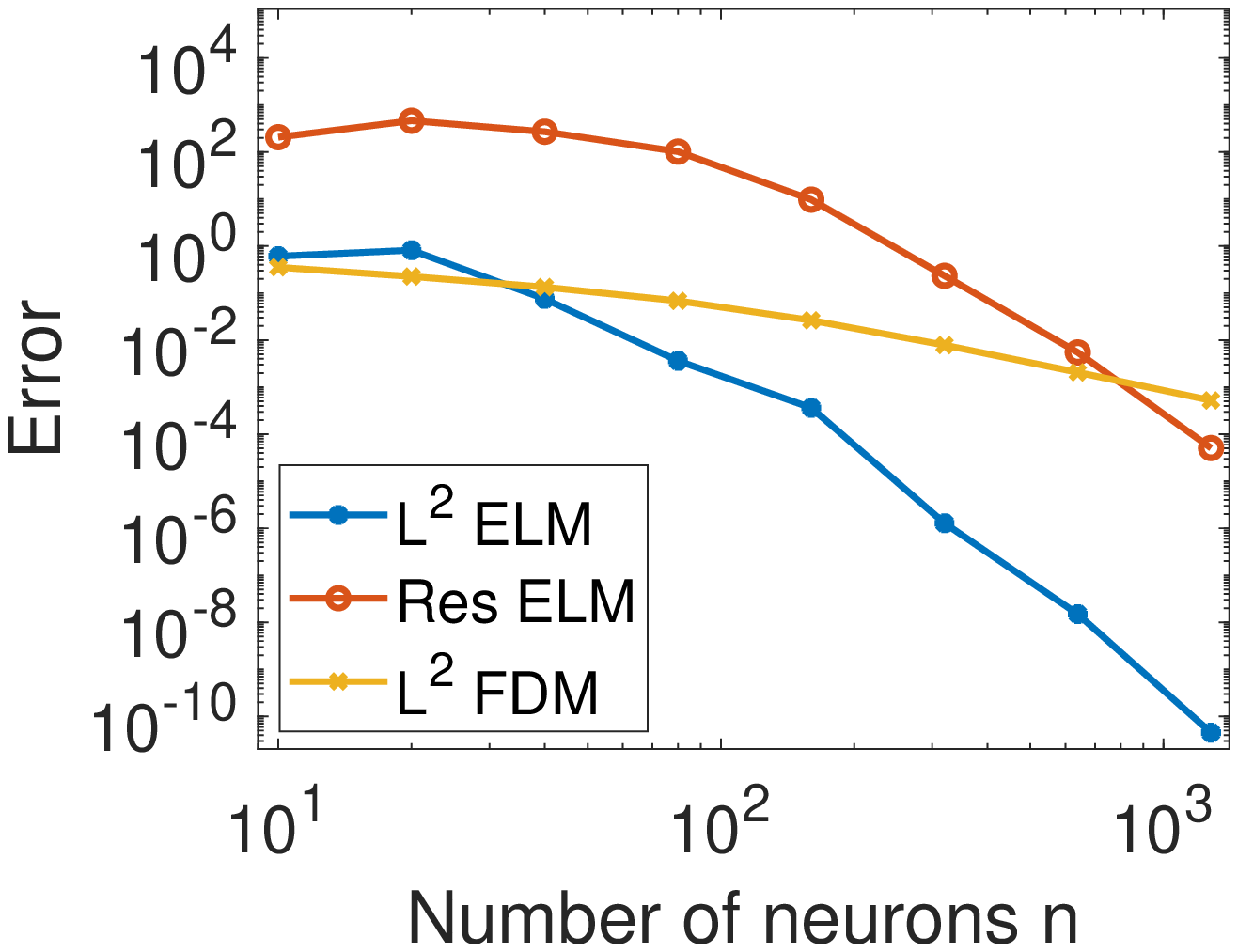}
\,
\includegraphics[width=0.45 \textwidth]{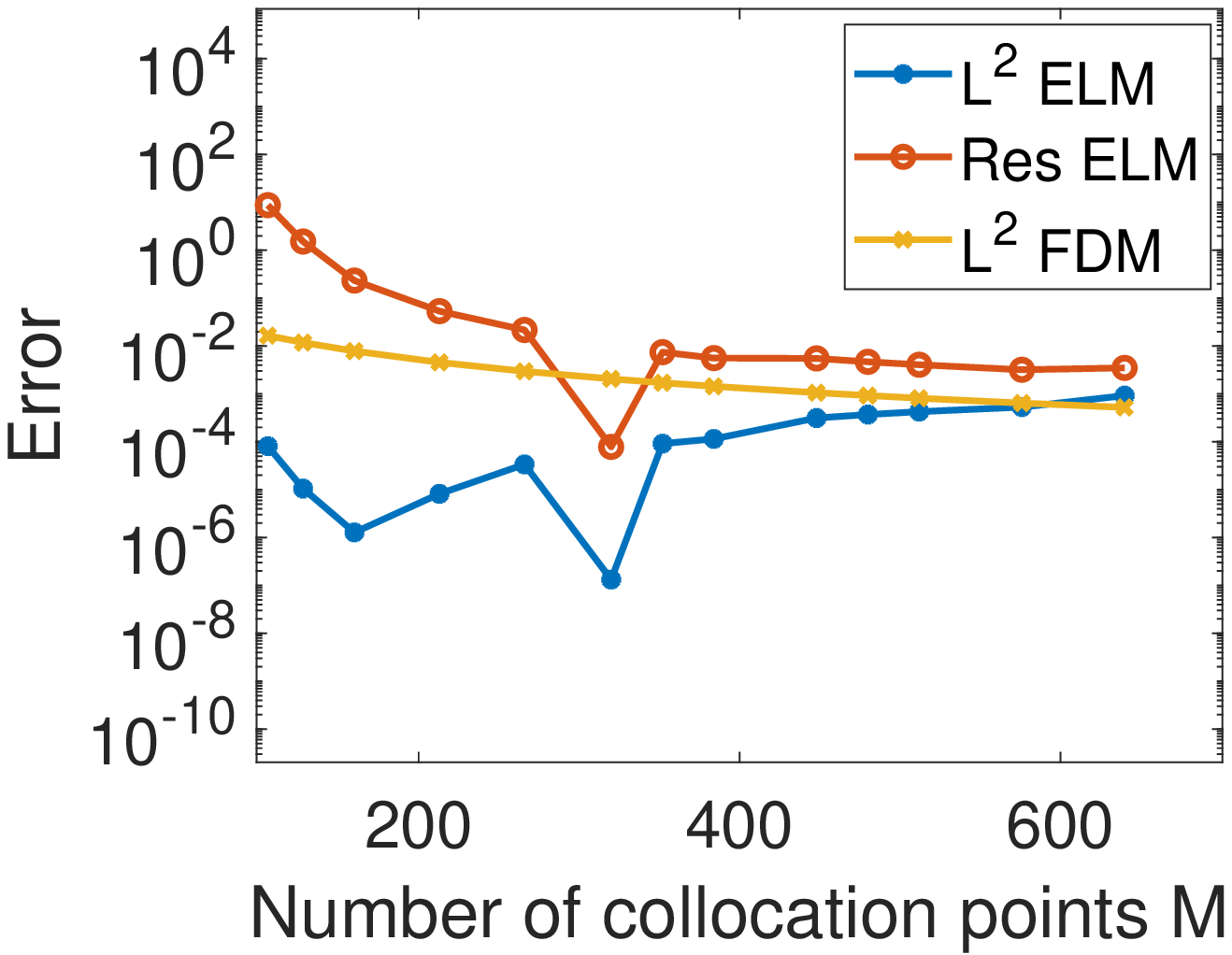}
\caption{Numerical tests for the solution of the problem \eqref{eq:ProDiff} with $\mu=1, \gamma=100, \lambda=0$. In this case, the P\'eclet number is $\mathbb{P}e_g=50$, the boundary layer is of the order of $10^{-2}$. On the top panel are shown both the exact \eqref{eq:exact_1} and the ELM numerical solutions with $n=40,80$; the number of collocation points is $M=n/2$. On the bottom panels are shown the error and residual convergence plots; on the left panel we fix $M=n/2$ and vary $n$ and on the right panel we fix $n=320$ and vary $M$.}
\label{Test:TraspDom_err}
\end{figure}

We considered the case with $\gamma=100$ and $\mu=1$, with a P\'eclet number $\mathbb{P}e_g=50$. Figure \ref{Test:reg3_err} depicts the exact analytical solution and some of the ELM numerical solutions along with the corresponding approximation errors. When $\gamma=0$, $\lambda>0$ and $f\equiv0$ in the problem \eqref{eq:ProDiff}, we have a diffusion-reaction  problem. As before, we impose Dirichlet boundary conditions with $g_0=0\,,\ g_1=1$. The problem has analogous difficulties as the previous one. The exact analytical solution is:
\begin{equation}
u(x)=\frac{\text{sinh}(\theta x)}{\text{sinh}(\theta)}, \qquad \text{ where } \theta=\sqrt{\lambda/\mu}.
\label{eq:exact_2}
\end{equation}
Also in this case the solution can give steep gradients if $\lambda/\mu>>1$. We call this behavior as \textit{boundary layer} at the right boundary of width order $\mathcal{O}(\sqrt{\mu/\lambda})$. In this case, the usual definition of global P\'eclet number is:
\begin{equation}
\mathbb{P}e_g=\frac{|\lambda|\cdot|I|^2}{6\mu} \ .
\label{eq:Pe_2}
\end{equation}
We have set $\lambda=300$ and $\mu=1$ so that the P\'eclet number is $\mathbb{P}e_g=50$ as before. The computed errors with respect to the exact solution are reported in Figure \ref{Test:ReazDom_err}. 

\begin{figure}[th]
\centering
\includegraphics[width=0.45 \textwidth]{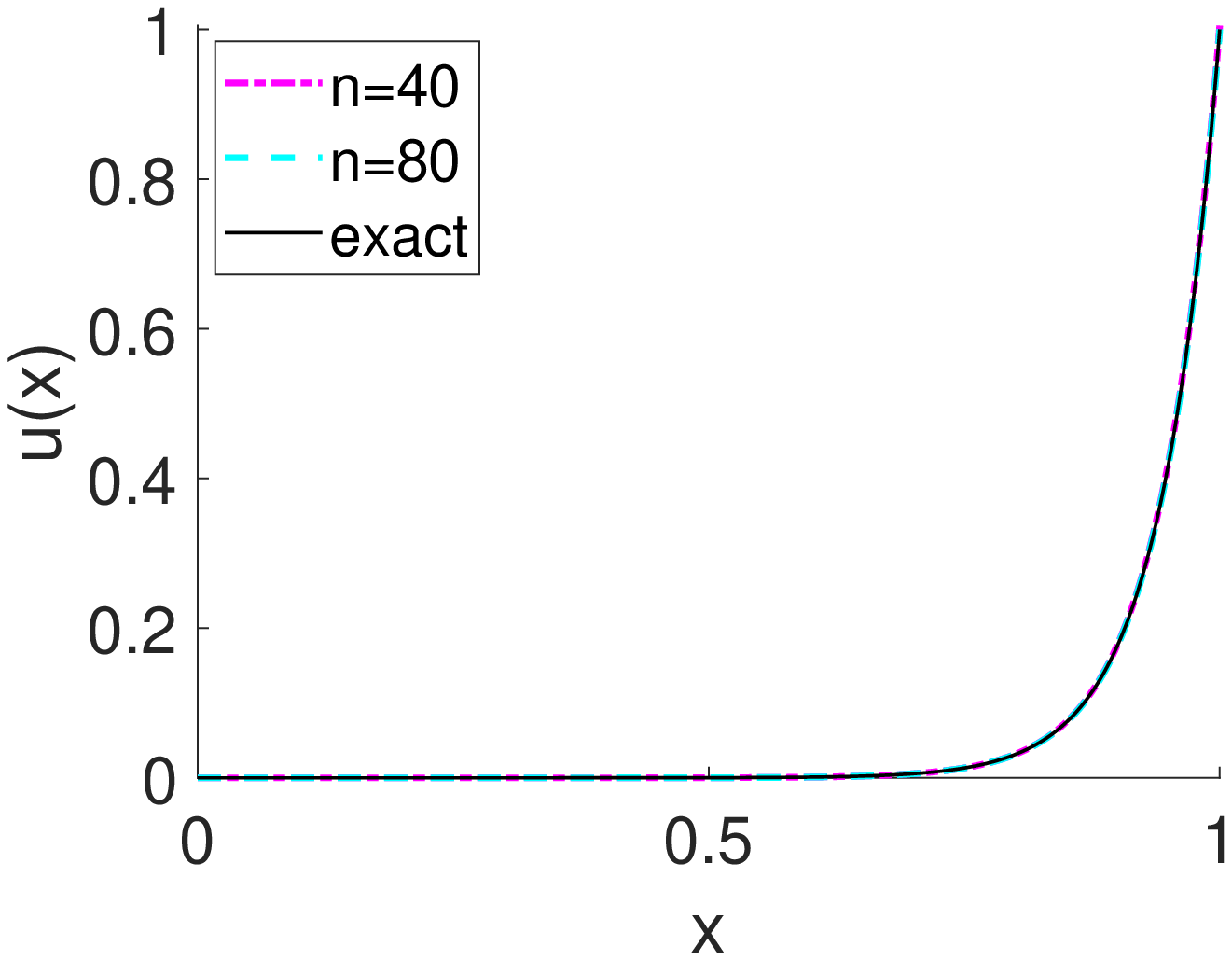}\\
\includegraphics[width=0.45 \textwidth]{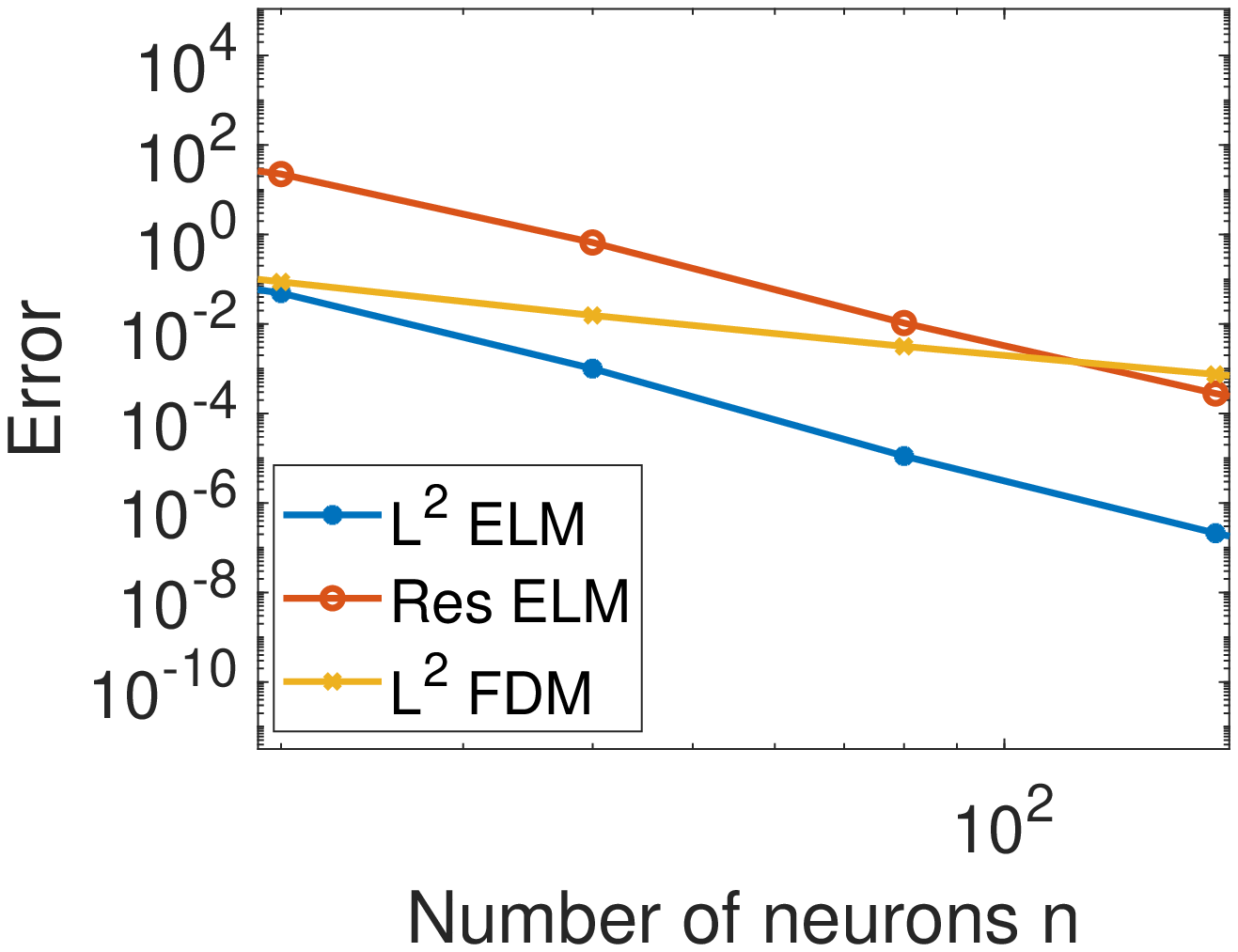}
\,
\includegraphics[width=0.45 \textwidth]{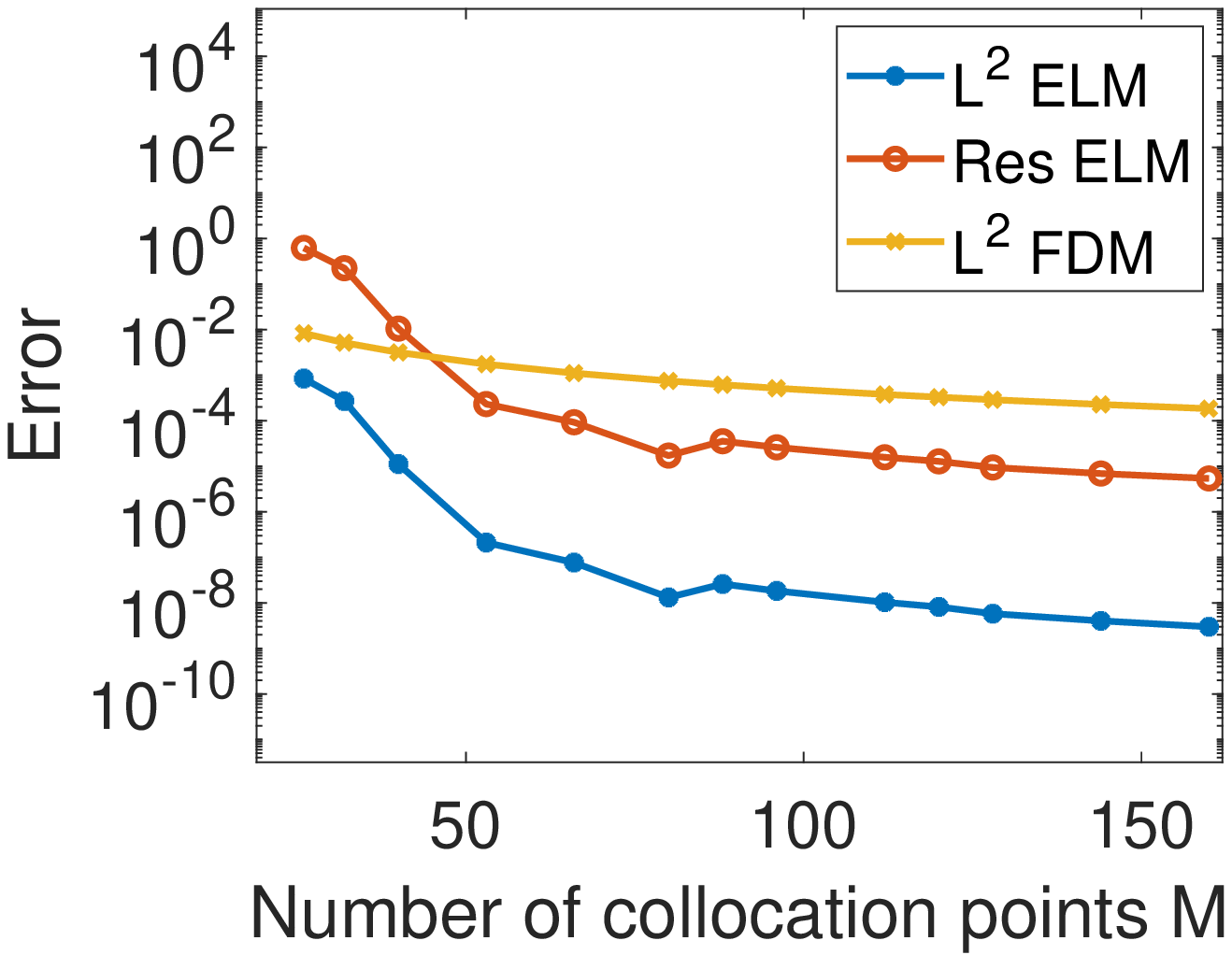}
\caption{Numerical tests for the solution of the problem \eqref{eq:ProDiff} with $\mu=1, \gamma=0, \lambda=300$. In this case, the P\'eclet number is $\mathbb{P}e_g=50$, the boundary layer is of the order of $10^{-1}$. On the top panel are shown the exact analytical \eqref{eq:exact_2} and the ELM numerical solutions for $n=40,80$; the number of collocation points was set to $M=n/2$. On the bottom panels are depicted the error and the residual convergence: on the left panel we fix $M=n/2$ and vary $n$ and on the right panel we fix $n=80$ and vary $M$.}
\label{Test:ReazDom_err}
\end{figure}
\begin{figure}[ht]
\centering
\includegraphics[width=0.45 \textwidth]{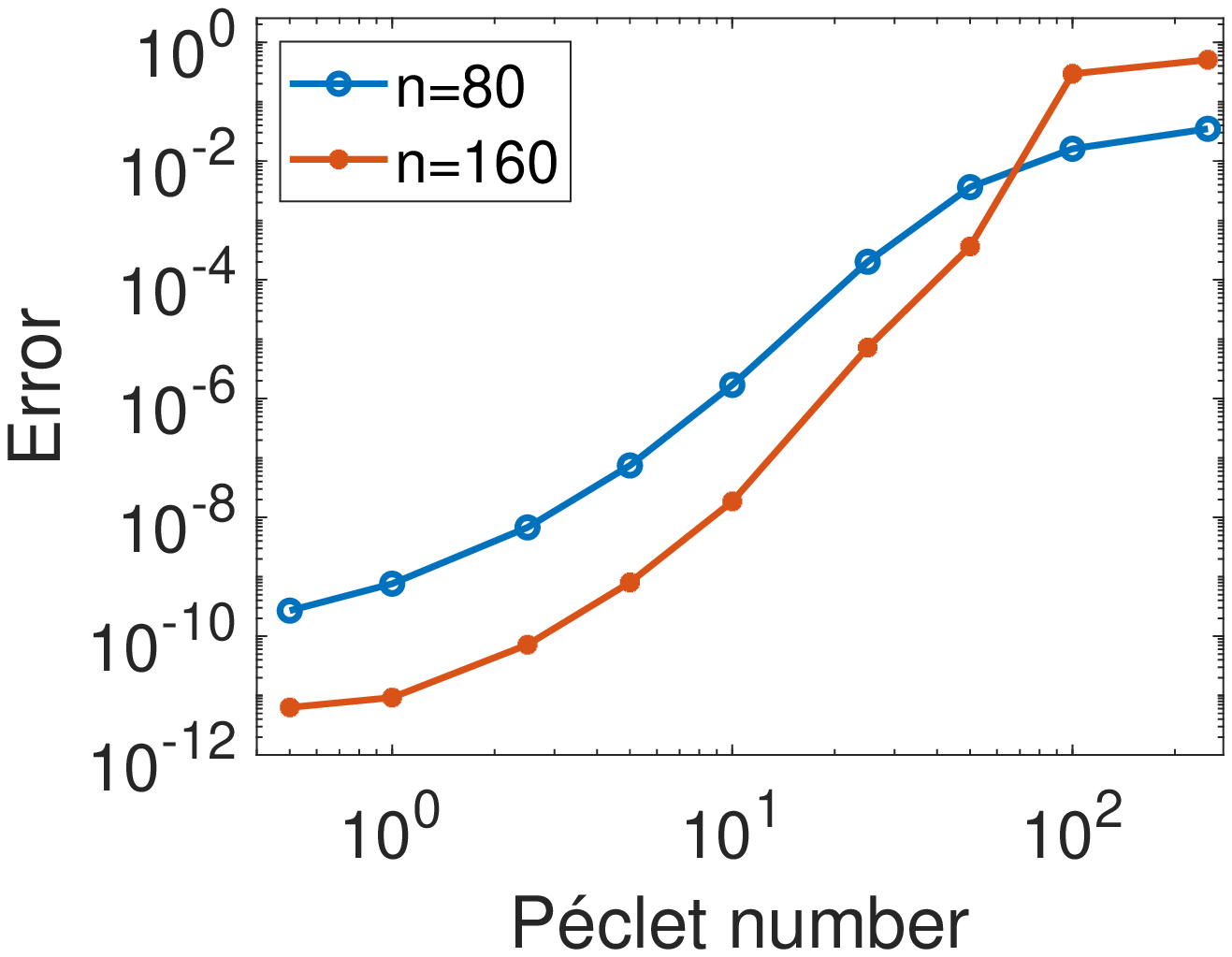}
\,
\includegraphics[width=0.45 \textwidth]{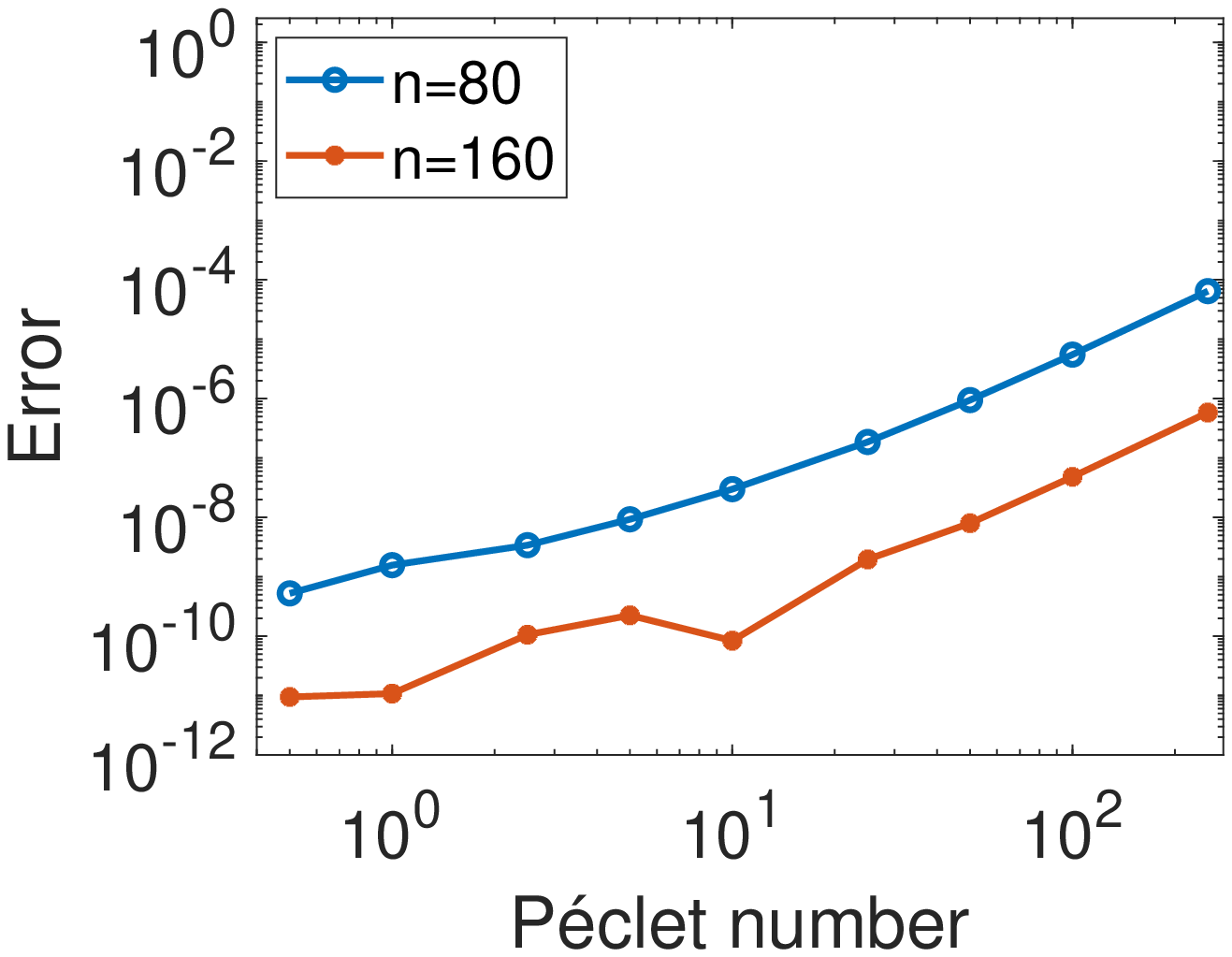}
\caption{$L^2$ errors of computed ELM solutions with fixed $n$ and $M=n/2$. On abscissae, we report the P\'eclet number as computed with equation \eqref{eq:Pe_1} on the left and equation \eqref{eq:Pe_2} on the right. In particular, on the left panel we consider $\gamma=[0,1,2,5,10,20,50,100,200,500]$ and on the right panel $\lambda=[0,3,6,15,30,60,150,300,600,1500]$.}
\label{Test:Peclet}
\end{figure}

Finally, we consider both the advection- and the reaction- dominated
problems taking different $\gamma$ and $\lambda$ so that we get appropriate values of the P\'eclet number. In Figure \ref{Test:Peclet}, we plot the $L^2$ errors of the computed ELM solutions versus the P\'eclet numbers for two different choices of number of neurons.

\subsection{Internal Layer Problems}\label{sec:internal}
In this section, we consider high transient problems that lead to an internal layer. For the solutions of these problems  an adaptive mesh is usually proposed. For a description of the numerical problems that can arise, we refer to \cite{MITCHELL2013350}. In this section, we show that the ELM network can provide good approximations. First, we consider the $atan$ problem:

\begin{equation}
\left \{ \begin{array}{l}
-u''=\frac{(2\alpha^3(x-x_0)}{(1+\alpha^2(x-x_0)^2)^2}  , \qquad 0<x<1\\
u(0)=\theta_0\,,\ u(1)=\theta_1
\end{array} \right.
\label{eq:peak_1}
\end{equation}
$\theta_0,\theta_1$ are fixed in order to have an exact solution that reads:
\begin{equation} 
u(x)=\text{atan}(\alpha(x-x_0))
\label{eq:sol_atan}
\end{equation}
In our tests, we fix $\alpha=60$ and $x_0=4/9$ that leads to a non-symmetric internal layer as in \cite{demkowicz2002fully}. 
Figure \ref{Test:Arctan_err} depicts the exact-analytical solution and some of the numerical ones. The approximation errors with respect to the exact-analytical solution are also given.

\begin{figure}[th]
\centering
\includegraphics[width=0.45 \textwidth]{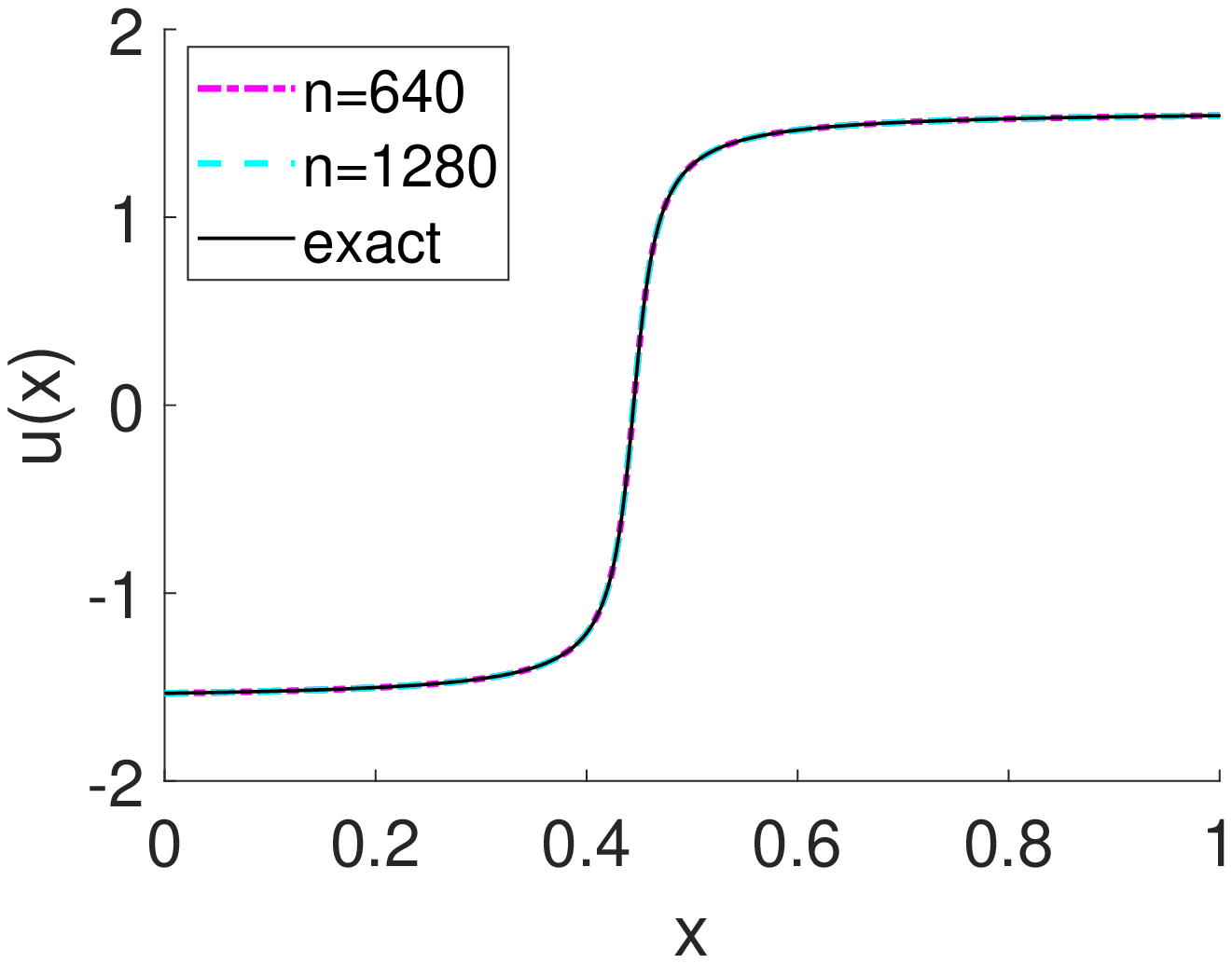}\\
\includegraphics[width=0.45 \textwidth]{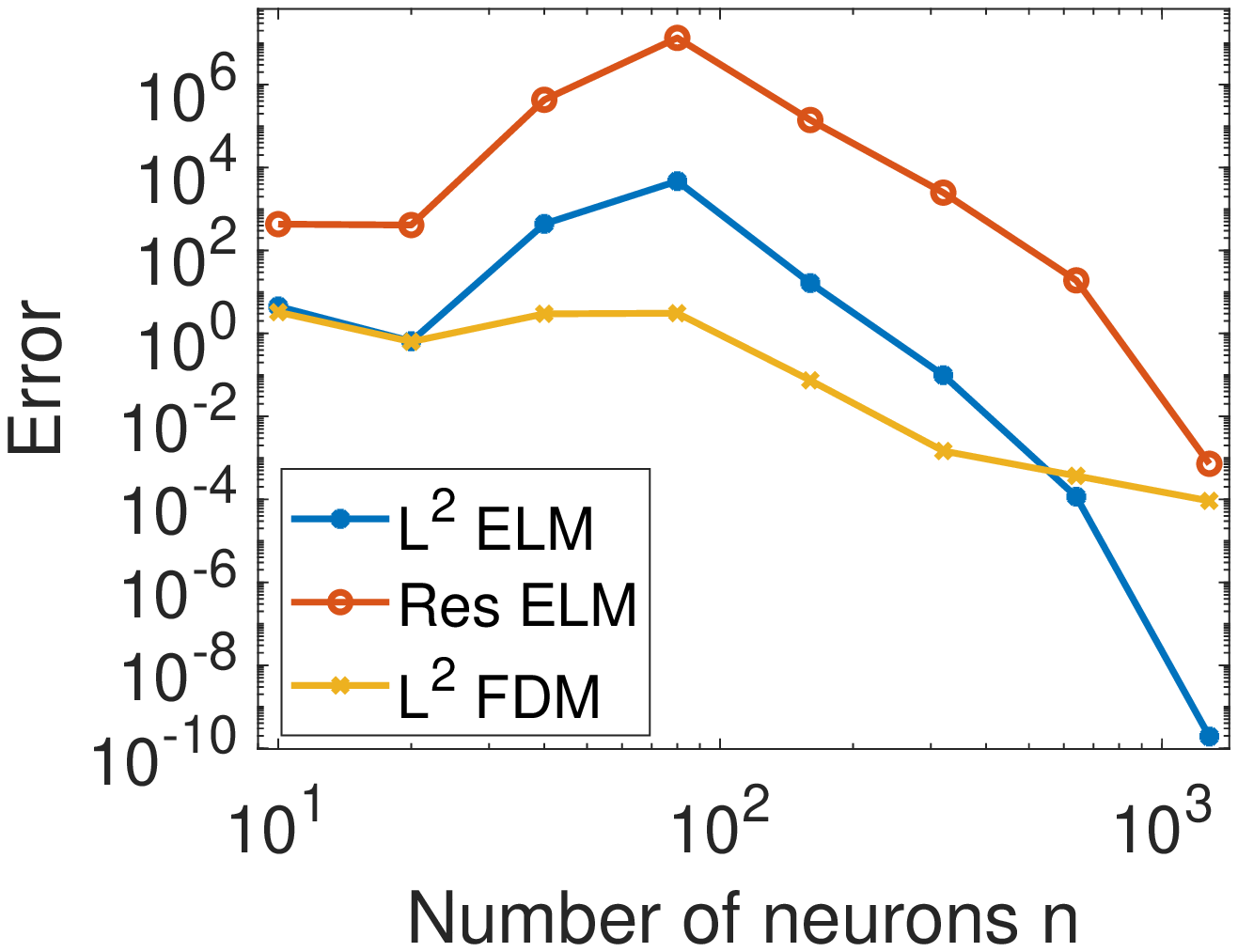}
\,
\includegraphics[width=0.45 \textwidth]{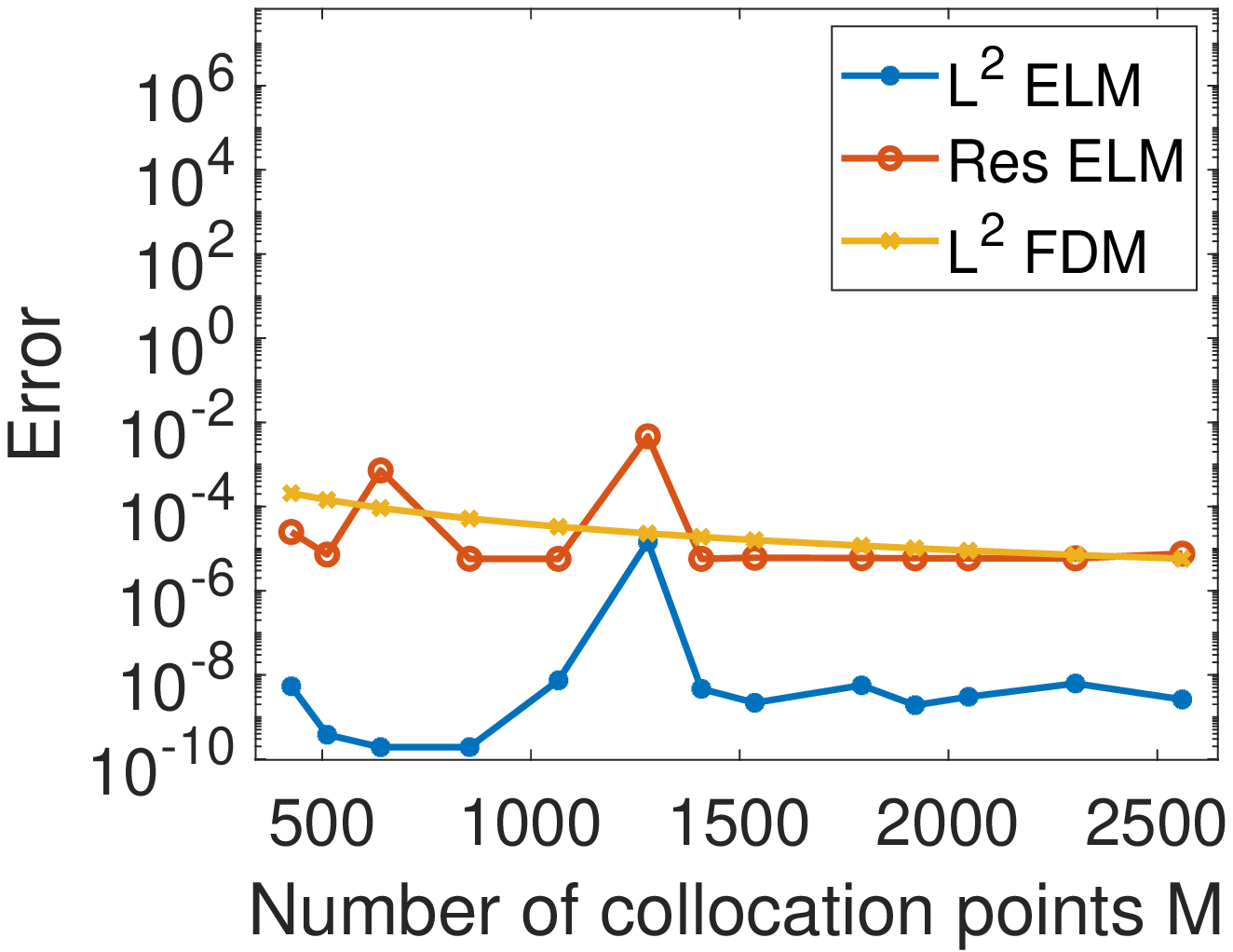}
\caption{
Numerical tests for the solution of problem \eqref{eq:peak_1} with $\alpha=60$ and $x_0=4/9$. On the top panel it is depicted the exact-analytical \eqref{eq:sol_atan} and the ELM numerical solutions for $n=640,1280$; the number of collocation points $M$ was fixed to $n/2$. On the bottom panel, we depict the error and residual convergence: on the left panel, we fix $M=n/2$ and vary $n$ and on the right panel we fix $n=1280$ and vary $M$.}
\label{Test:Arctan_err}
\end{figure}

The second test with internal peak is a rescaled sinusoidal problem suggested in \cite{schmidt2000posteriori}:
\begin{equation} 
\left \{ \begin{array}{l}
-u''=\big(\frac{-4x^2}{\varepsilon^2}+\frac{2}{\varepsilon}\big)e^{-x^2/\varepsilon}   , \qquad 0<x<1\\
u(0)=\theta_0\,,\ u(1)=\theta_1
\end{array} \right.
\label{eq:peak_2}
\end{equation}
$\theta_0,\theta_1$ are fixed in order to have as exact solution: 
\begin{equation}
u(x)=e^{-x^2/\varepsilon} \ .
\label{eq:exact_peak_2}
\end{equation}
We consider the case $\varepsilon=10^{-3}$. Figure \ref{Test:Peak_err} depicts the exact-analytical solution and  some of the numerical ones. The approximation errors with respect to the exact-analytical solution are also given.

\begin{figure}[th]
\centering
\includegraphics[width=0.45 \textwidth]{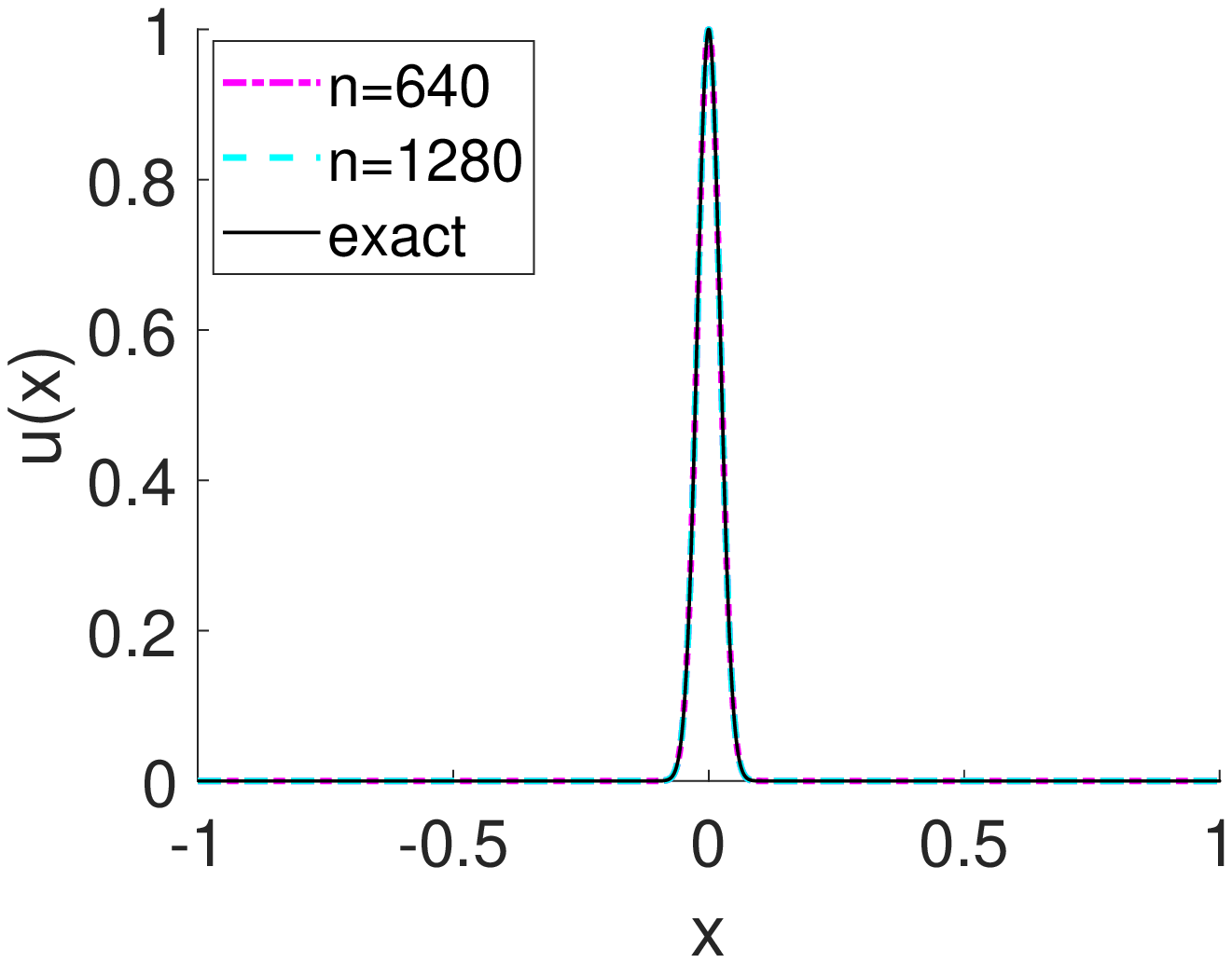}\\
\includegraphics[width=0.45 \textwidth]{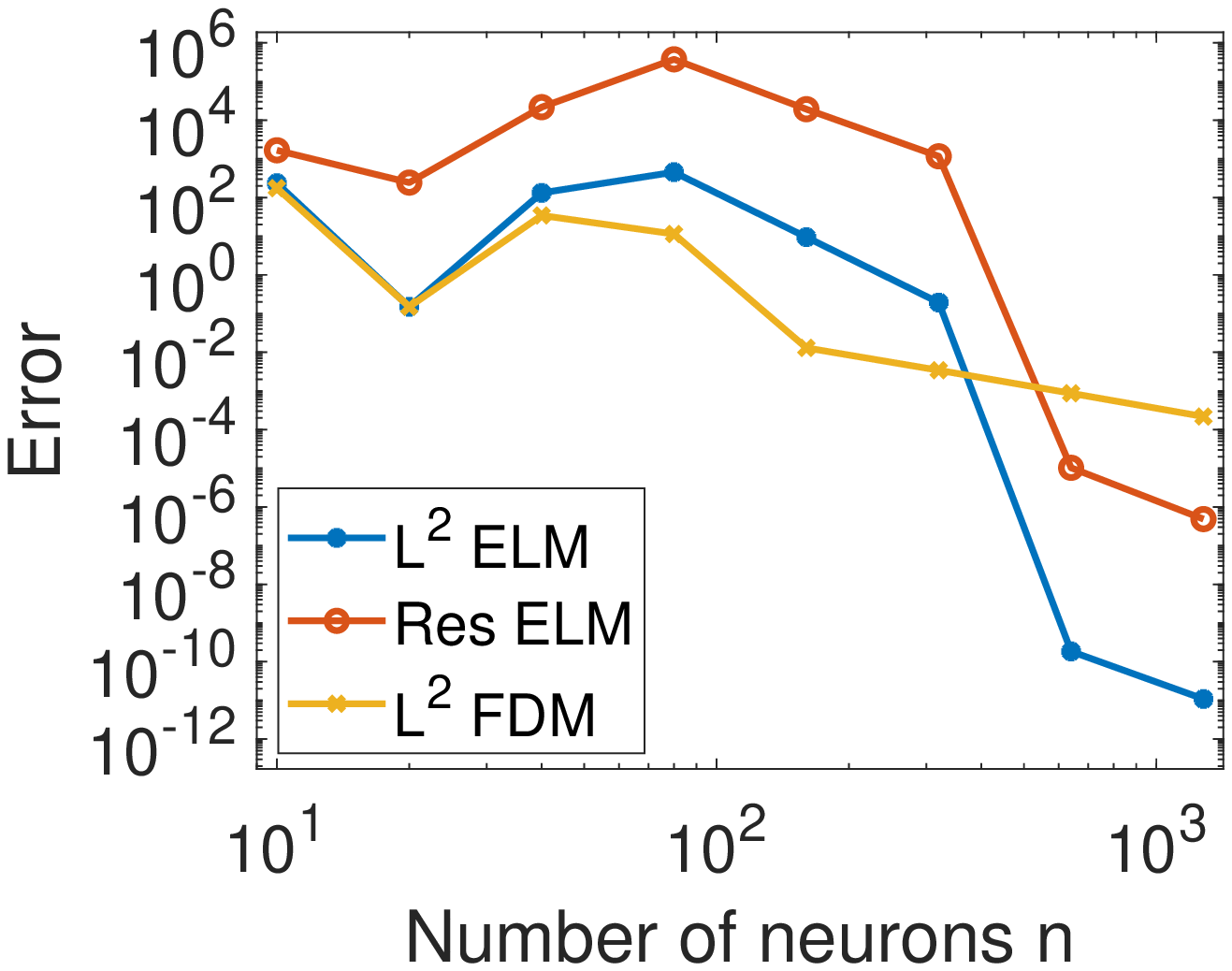}
\,
\includegraphics[width=0.45 \textwidth]{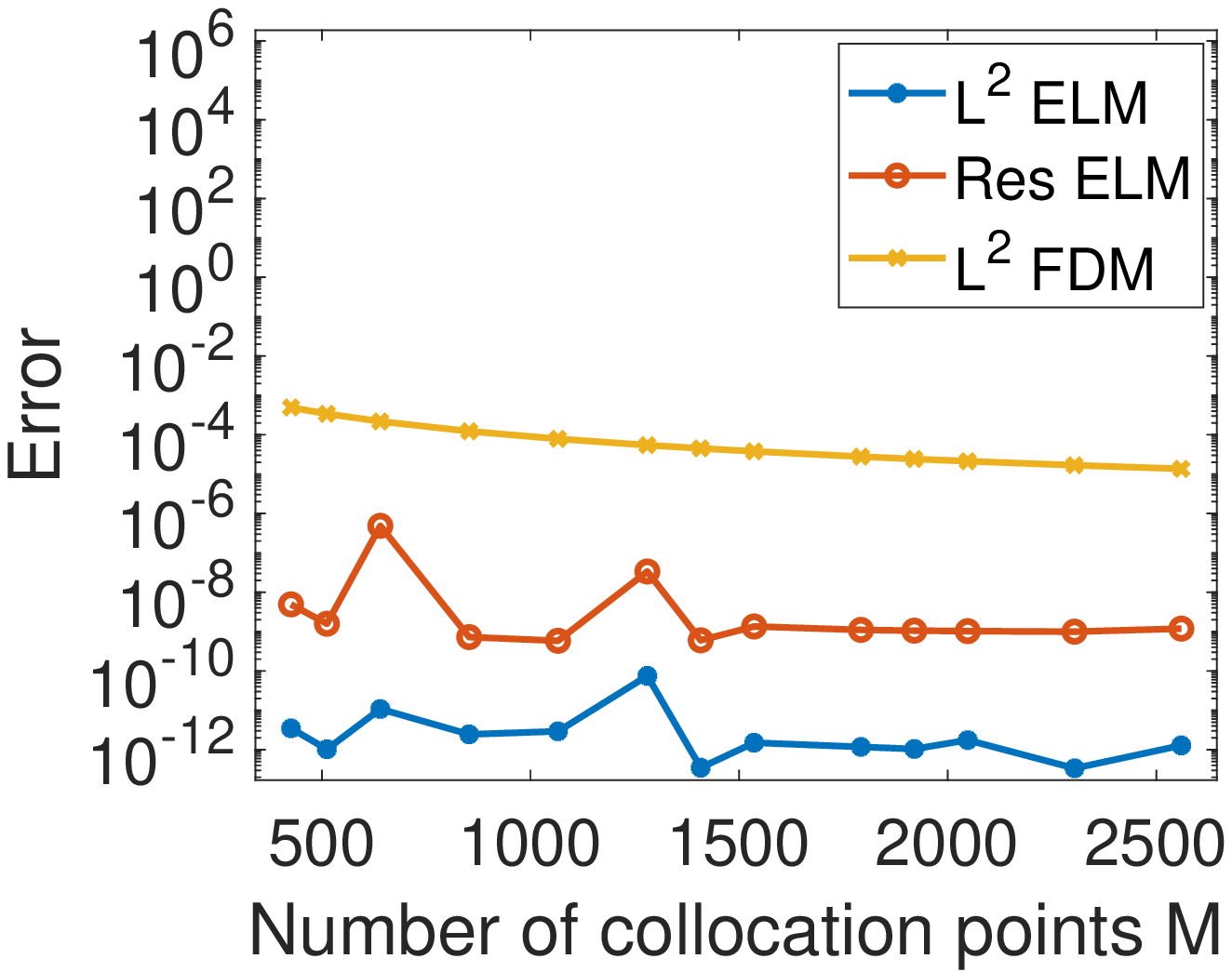}

\caption{Numerical results for the solution of the problem \eqref{eq:peak_2} with $\varepsilon=10^{-3}$. On the top panel are depicted the exact-analytical \eqref{eq:exact_peak_2} and the computed ELM solutions with number of neurons $n=640,1280$ and number of collocation points $M=n/2$. On the bottom panel, are depicted the error and residual convergence: on the left panel, we fix $M=n/2$ and vary $n$ and on the right panel, we fix $n=1280$ and vary $M$.}
\label{Test:Peak_err}
\end{figure}

To this end, we consider a very difficult numerical problem that has as exact-analytical solution a comb-like profile. For this problem, we compute the approximation errors using several number of neurons and the behavior only on collocation points. The equation we consider is the following:
\begin{equation}
\left \{ \begin{array}{l}
-u''=-2(\varepsilon+x)\cos\bigg(\frac{1}{\varepsilon+x}\bigg)-\frac{\sin\big(\frac{1}{\varepsilon+x}\big)}{(\varepsilon+x)^4} , \qquad 0<x<1\\
u(0)=\theta_0\,,\ u(1)=\theta_1
\end{array} \right.
\label{eq:oscil}
\end{equation}
We have set $\theta_0,\theta_1$ in order to have the following exact-analytical solution:
\begin{equation}
u(x)=\sin\Biggl(\frac{1}{\varepsilon+x}\Biggr) \ .
\label{eq:exact_oscil}
\end{equation}
Then, we considered the case $\varepsilon=1/10\pi$ that gives a solutions with 10 zeros and 5 oscillations in the domain. Figure \ref{Test:Oscillatory_err} depicts the exact-analytical solution as well as the ELM numerical solution for $n=1280, 2560$ and the resulting approximation errors with respect to the exact-analytical solution considering only the approximation errors at the collocation points using the $L^\infty$ norm.

\begin{figure}[th]
\centering
\includegraphics[width=0.45 \textwidth]{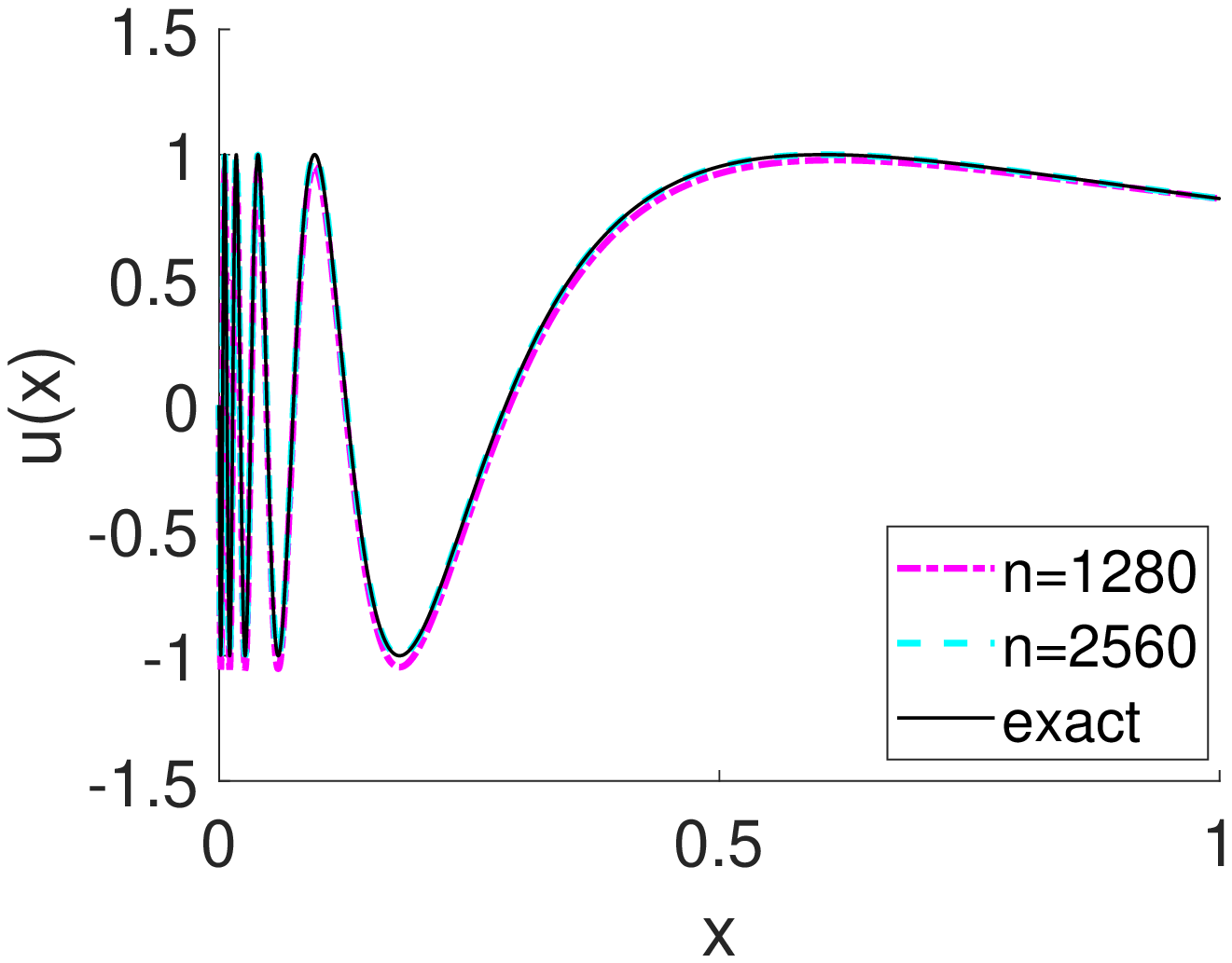}\\
\includegraphics[width=0.45 \textwidth]{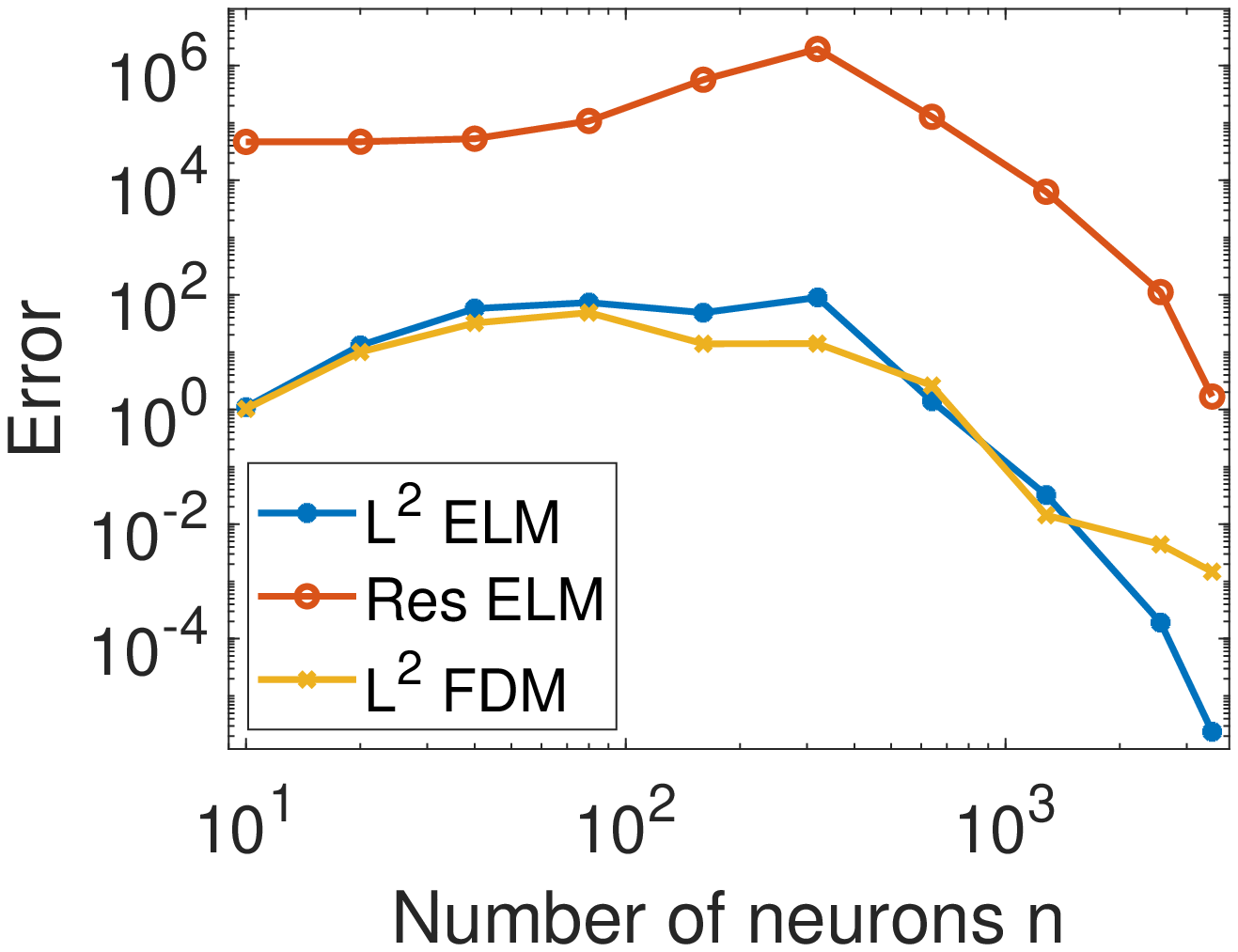}
\,
\includegraphics[width=0.45 \textwidth]{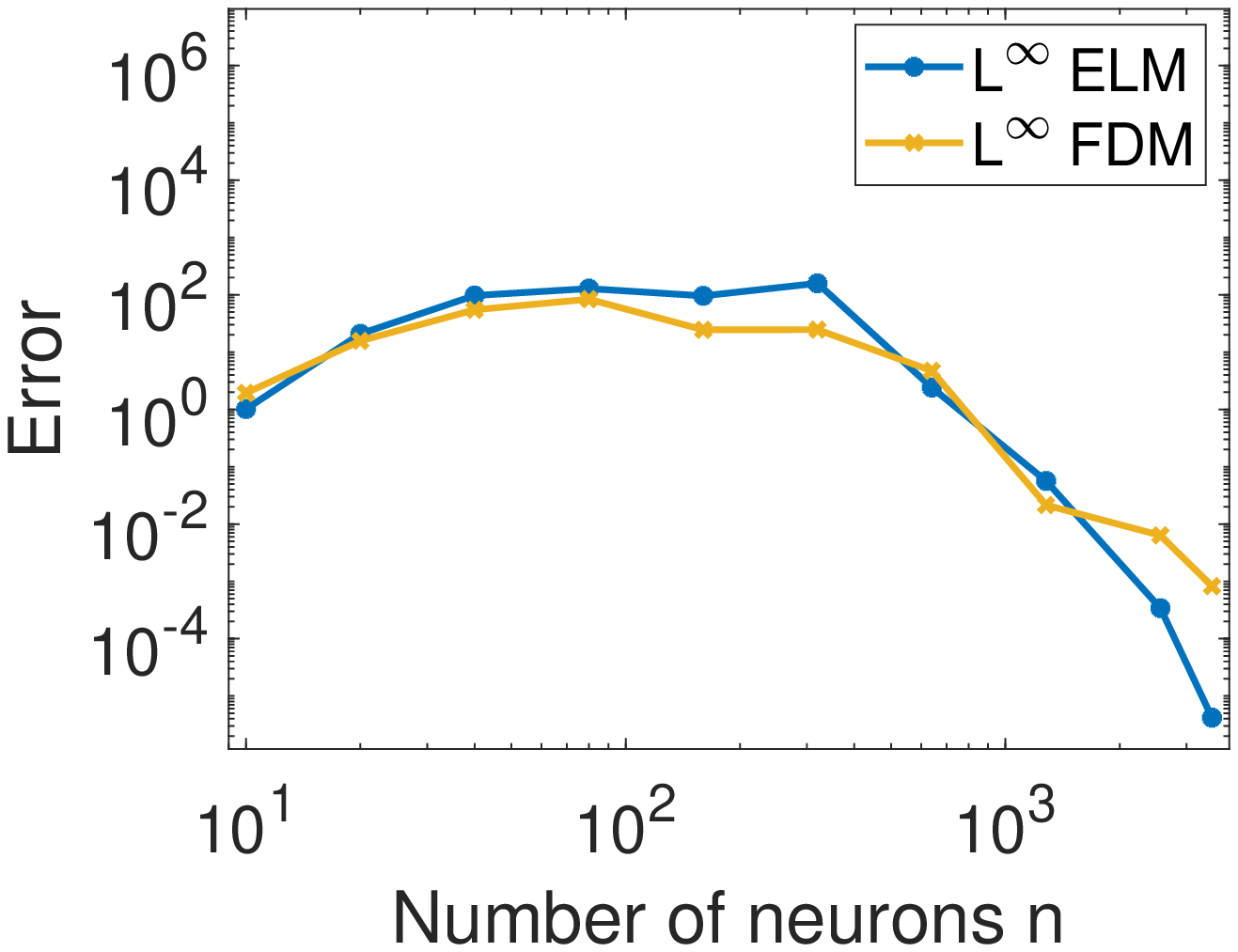}
\caption{Numerical results for the solution of the problem \eqref{eq:oscil} with $\varepsilon=1/10\pi$. On the top panels it is depicted the exact-analytical \eqref{eq:exact_oscil} and the numerical solutions with the number of neurons $n=1280,2560$ and number of collocation points fixed at $M=n/2$. On the bottom panels are shown the error and residual convergence: on the left panel, it is shown the  $L^2$ norms with fixed $M=n/2$ and on the right panel it is shown the $L^\infty$ on the collocation points.}
\label{Test:Oscillatory_err}
\end{figure}

\section{Conclusions and future work} 
In the present paper, we address collocation scheme based on ELM networks for the numerical approximation of 1D elliptic problems with steep gradients. The weights between the input layer and the hidden layer as well as the biases are fixed appropriately in advance and kept fixed. In this way, the only coefficients that describe the approximation space are the weights of the linear combination used for the element of the linear activation of the ELM. First, we have explored the underlying space, so as to construct it in a random but proper way. Then, the numerical solution to the boundary value problems was reduced to the computation of the $n$ weights of the ELM network via $M$ collocation equations. We have proved the convergence of the scheme in the square case, and proposed the use of an underdetermined collocation scheme to avoid ill-conditioning.\par
Numerical tests were performed for the evaluation of the approximation accuracy of the proposed method with the aid of benchmark boundary value problems exhibiting sharp gradients. We also compared the proposed approach with a high order finite difference (FD) scheme. We show that for both internal and boundary layer problems the proposed method outperforms FD, when the number of neurons is taken to be big enough to catch the steep gradient of the layer. We emphasize also that in our case, the solutions are $C^\infty$ functions.\par
The proposed method can be further developed by exploring some alternatives, such as the use of optimal activation functions, as proposed in \cite{guliyev2018approximation}, or the selection of the solution of the under-determined system by using the properties of null rules, as explored in \cite{calabro2019null}. Finally, the application of the proposed ELM collocation method can be extended in various ways. Among others, we are currently considering the use for:
\begin{itemize}
\item Non linear differential problems;
\item Multidimensional problems;
\item Time-dependent problems;
\item Inverse problems with overfitting.
\end{itemize}

%

\section*{Acknowledgements}
The authors would like to thank Prof. Ferdinando Auricchio, Prof. Bert J\"{u}ttler for fruitful discussions on the subject of the paper.
Francesco Calabr\`o are partially supported by INdAM, through GNCS research projects. This support is gratefully acknowledged.

 \bibliographystyle{abbrv}
 \bibliography{AnnPDE_Bib}

\end{document}